\documentclass[11pt,a4paper,pdftex]{amsart}

\usepackage{amssymb} 
\usepackage{amsmath} 
\usepackage{amsthm} 
\usepackage{ascmac} 
\usepackage{bm} 
\usepackage{color} 
\usepackage[OT2,T1]{fontenc} 
\usepackage{fullpage}
\usepackage[dvipdfmx]{graphicx} 
\usepackage{here}
\usepackage{indentfirst}
\usepackage{mathrsfs} 
\usepackage{mathtools} 
\usepackage{multirow}
\usepackage{tikz} 
\usetikzlibrary{cd}

\allowdisplaybreaks[4]
\mathtoolsset{showonlyrefs=true}

\theoremstyle{plain}
\newtheorem{thm}{Theorem} 
\newtheorem{prop}[thm]{Proposition}
\newtheorem{lem}[thm]{Lemma}

\newtheorem{cor}[thm]{Corollary}
\numberwithin{thm}{section}
\theoremstyle{definition}

\newtheorem{rem}[equation]{Remark}
\newtheorem*{ac}{Acknowledgements} 

\makeatletter
\newcommand{\MyAlphabet}[1]{\@tfor\Ch@r:=#1\do{
	\expandafter\edef\csname bb\Ch@r\endcsname{\noexpand\mathbb{\Ch@r}}
	\expandafter\edef\csname cal\Ch@r\endcsname{\noexpand\mathcal{\Ch@r}}
	\expandafter\edef\csname frak\Ch@r\endcsname{\noexpand\mathfrak{\Ch@r}}
	\expandafter\edef\csname scr\Ch@r\endcsname{\noexpand\mathscr{\Ch@r}}
}}
\makeatother
\MyAlphabet{ABCDEFGHIJKLMNOPQRSTUVWXYZabcdefghijklmnopqrstuvwxyz}
\makeatletter
\newcommand{\MyMathOperators}[1]{\@for\op:=#1\do{
	\expandafter\edef\csname\op\endcsname{\noexpand\mathop{\noexpand\mathrm{\op}}\nolimits}
}}
\makeatother
\let\Re\relax
\let\Im\relax
\MyMathOperators{%
	rank,%
    Re,%
    Im,%
    sgn,%
    Gal,%
}

\renewcommand{\epsilon}{\varepsilon}
\renewcommand{\bar}[1]{\overline{#1}}
\renewcommand{\tilde}[1]{\widetilde{#1}}

\DeclarePairedDelimiter{\skakko}{\lparen}{\rparen} 
\DeclarePairedDelimiter{\mkakko}{\lbrace}{\rbrace} 
\DeclarePairedDelimiterX{\set}[2]{\lbrace}{\rbrace}{#1\,\delimsize\vert\,#2} 

\DeclareSymbolFont{cyrletters}{OT2}{wncyr}{m}{n}
\DeclareMathSymbol{\Sha}{\mathalpha}{cyrletters}{"58}

\usepackage{hyperref}
\hypersetup{colorlinks=true,urlcolor=blue,citecolor=blue,linkcolor=blue}

\title{Lower bound for the 2-adic valuations of central $L$-values of elliptic curves with complex multiplication}
\author{Keiichiro Nomoto}
\address{FACULTY OF MATHEMATICS, KYUSHU UNIVERSITY, MOTOOKA 744, NISHI-KU FUKUOKA 819-0395, JAPAN}
\email{nomotokeiichiro@gmail.com}
\date{}

\begin{document}

\begin{abstract}
    Let $E_{-D}$ be the elliptic curve $y^2=x^3+Dx$
    defined over $K=\bbQ(i)$ for $D\in K$ which is coprime to $2$.
    In this paper, we give a lower bound for the $2$-adic valuation of the
    algebraic part of the central value of Hecke $L$-function associated to $E_{-D}$.
\end{abstract}

\maketitle

\section{Introduction}\label{sec:Introduction}

Let $E$ be an elliptic curve defined over an imaginary quadratic field $K$
with complex multiplication by the integer ring $\calO_K$
and let $L(\psi, s)$ be the Hecke $L$-function attached to $E$.
For a CM period $\Omega$, it is known that the value
\begin{align}
    \frac{L(\psi, 1)}{\Omega}
\end{align}
belongs to $K$. Calculating the $p$-adic valuation of this value is essentially
equivalent to computing the order of the $p$-primary part of the Tate--Shafarevich group.
Although it is difficult to calculate it for a general prime number $p$,
it is slightly easier to compute the $2$-adic valuation
if the elliptic curve $E$ is a twist of some elliptic curve.

Several works are giving lower bounds of the $p$-adic valuations
of various families of elliptic curves with complex multiplication
when $p=2, 3$.
First, we give some results for elliptic curves of the form $y^2=x^3-Dx$.
Zhao has given a lower bound of the 2-adic valuations
when $D=(\pi_1\cdots \pi_n)^2\in\bbZ[i] \ (\pi_i\equiv 1\bmod 4)$
is the square of the product of distinct Gaussian primes in \cite{Zhao:1997}
and when $D=(p_1\cdots p_m)^2 \ (p_i\equiv 1\bmod 8)$
is the square of the product of distinct rational primes in \cite{Zhao:2001}.
He has also given it for $D=4(\pi_1\cdots\pi_n)^2
\ (\pi_i\equiv 1\bmod 2+2i)$ in \cite{Zhao:2003}.
Qiu and Zhang \cite{QiuZhang:2002} have given a lower bound of the 2-adic valuations
for $D=\pi_1\cdots\pi_n,\ (\pi_1\cdots \pi_r)^2\pi_{r+1}\cdots \pi_n \ (\pi_i\equiv 1\bmod 4)$.
In the latter case $D=(\pi_1\cdots \pi_r)^2\pi_{r+1}\cdots \pi_n$,
however, no proof has been given.
Next, we give some results for elliptic curves of the form $y^2=x^3-2^43^3D$.
Qiu and Zhang \cite{Qiu:2002} have given a lower bound of the 3-adic valuations
when $D=(\pi_1\cdots\pi_n)^2\in\bbZ[\omega] \ (\pi_i\equiv 1\bmod 6)$
is the square of the product of distinct Eisenstein primes.
Qiu \cite{Qiu:2003} also has given it for $D=(\pi_1\cdots \pi_n)^{4} \ (\pi_i\equiv 1\bmod 6)$
and for $D=(\pi_1\cdots \pi_n)^3 \ (\pi_i\equiv 1\bmod 12)$.
Kezuka \cite{Kezuka:2021} has given a lower bound of the 3-adic valuations
for the elliptic curve $y^2=x^3-2^43^3D^2$ defined over $\bbQ$
when $D$ is a cube-free integer with $(D, 3)=1$.
There are also some studies on CM elliptic curves
with these $j$-invariants being not 0 or 1728
(\textit{cf}. \cite{Coates:2015}, \cite{Coates:2014}, \cite{Choi:2019}).
\\


Let $K=\bbQ(i)$. We consider the elliptic curve $E_{-D}: y^2=x^3+Dx$
defined over $K$ for $D\in K$ which is coprime to $2$.
We write the Hecke character associated to $E_{-D}$ as $\psi_{-D}$.
We give a lower bound for the $2$-adic valuation
of the algebraic part of $L(\overline{\psi_{-D}}, 1)$.
The following theorem is the main result. 

\begin{thm}\label{thm:mainthm}
    Suppose $D\in \calO_K$, quartic-free, and is congruent to $1$ modulo $2+2i$.
    Let $\psi_{-D}$ be the Hecke character associated to the elliptic curve $E_{-D}: y^2=x^3+Dx$ defined over $K$.
    We define $L_{2}(\overline{\psi_{-D}}, s)$ to be the Hecke $L$-function
    of $\overline{\psi_{-D}}$ omitting the Euler factor corresponding
    to the prime $(1+i)\calO_K$.
    If $D\notin K^{\times 2}$, then we have
    \begin{align}
    v_2\skakko*{\dfrac{L_{2}(\overline{\psi_{-D}}, 1)}{\Omega}}\geq \dfrac{r(D)-2}{2},
    \end{align}
    where $r(D)$ is the number of distinct primes in $D$,
    $\Omega=2.6220575\dots$ is the least positive real element
    of the period lattice of $E_1: y^2=x^3-x$
    and $v_2$ is the 2-adic valuation of $\overline{\bbQ_2}$
    normalized so that $v_2(2)=1$.
\end{thm}

\begin{rem}
    When $D\in K^{\times 2}$, Zhao has given the lower bound
    $(2r(D)-3)/2$ \cite[Theorem 1]{Zhao:2003}.
    Note that Zhao uses a period of $E_{4D}$,
    while we use a period $\Omega$ of $E_1$.
\end{rem}

\begin{rem}
    The condition that $D\in \calO_K$, quartic-free and congruent to $1$ modulo $2+2i$ in Theorem \ref{thm:mainthm} is not essential.
    If $D\in K$ is not quartic-free, then we can take $D_0\in \calO_K$ so that quartic-free and $E_{-D}$ is isomorphic to $E_{-D_0}$ over $K$.
    For any $D\in \calO_K$ which is coprime to 2, only one of $\{\pm D, \pm iD\}$ is congruent to 1 modulo $2+2i$.
    For more details, see Section \ref{sec:ExpressFiniteSum}.
\end{rem}

\begin{rem}
    The lower bound of Theorem \ref{thm:mainthm} is expected to be sharp
    in the sense that there exist elliptic curves $E_{-D}$ for which equality holds.
    See the numerical examples in Section \ref{sec:Example}.
\end{rem}


We prove Theorem \ref{thm:mainthm}
combining Theorem \ref{thm:Lvaluation} with Theorem \ref{thm:2Lvaluation}.
Here, Theorem \ref{thm:Lvaluation} deals with the case where all the indices of the primes in $D$ are equal,
and Theorem \ref{thm:2Lvaluation} deals with the other case.
The key of the proof of Theorem \ref{thm:Lvaluation} and Theorem \ref{thm:2Lvaluation}
is to consider not only an elliptic curve $E_{-D}$ for a parameter $D$
but also elliptic curves $E_{-D_T}$ for all divisors $D_T$ of $D$.
Theorem \ref{thm:Lvaluation} is proved
by using the induction on the number of primes dividing $D$.
Such a method is sometimes called Zhao's method.
However, due to technical reasons,
Zhao's method can only be applied to the case
where all the indices of the primes in $D$ are equal.
In order to apply Zhao's method to the other case,
we decompose $D$ into $D_1D_2D_3$,
where $D_i$ is the product of the primes whose indices are all equal to $i$.
By iterating Zhao's method for each $D_i$,
we evaluate the 2-adic valuation for general $D$ and prove Theorem \ref{thm:2Lvaluation}.

We deal only with the family $y^2=x^3+Dx$ in this paper.
However, the essence of the proof of Theorem \ref{thm:2Lvaluation}
is that $D$ can be uniquely decomposed into the product of primes in $K$.
Therefore, our iterative Zhao's method may be applicable to
CM elliptic curves defined over fields with class number one.
\\

After writing this paper, we noticed that Kezuka has also given a lower bound
of the 3-adic valuation for the elliptic curve $y^2=x^3-2^43^3D^2$
defined over $\bbQ$ using an iterative Zhao's method similar to ours
in the proof of \cite[Theorem 2.4]{Kezuka:2021}.
\\

In section \ref{sec:ExpressFiniteSum}, we write the $L$-value at $s=1$
as a finite sum using a special value of the Weierstrass $\wp$-function.
In section \ref{sec:2AdicValuation}, we evaluate the 2-adic valuation
of the $L$-value by using Zhao's method.
In the proof of Theorem \ref{thm:2Lvaluation}, we use Zhao's method iteratively.
For this reason, the proof is complicated,
and please refer to the inserted figures as necessary.
\section{Expression of \texorpdfstring{$L$}{L}-value as a finite sum}\label{sec:ExpressFiniteSum}

In this section, we write the $L$-value at $s=1$ as a finite sum
using a special value of the Weierstrass $\wp$-function.
Theorem \ref{thm:finite_sum} has already been proved
by Birch and Swinnerton-Dyer\cite{BirchSwinnerton-Dyer:1965};
however, for the readers convenience, we calculate it again.\\

Since $E_{-D}$ is isomorphic to $E_{-d^4D}$ over $K$ for $d\in K^\times$,
we may assume that $D\in \calO_K$ and quartic-free.
In the rest of this paper, we consider only the elliptic curve
$E_{-D}: y^2=x^3+Dx$ defined over $K$ for $D\in \calO_K$
that is coprime to $2$ and quartic-free.
Let $\psi_{-D}$ be the Hecke character of $K$ associated to $E_{-D}$
and let
\begin{align}
    \Omega=\int_{1}^{\infty}\frac{dx}{\sqrt{x^3-x}}=2.6220575\dots
\end{align}
be the least positive real element of the period lattice of $E_1: y^2=x^3-x$.
For a non-zero element $g\in \calO_K$,
$L_{g}(\overline{\psi}, s)$ denotes the Hecke $L$-function of $\overline{\psi}$
omitting all Euler factors corresponding to the primes that divide $g\calO_K$; that is;
\begin{align}
    L_{g}(\overline{\psi}, s)=L(\overline{\psi}, s)\prod_{\frakp|g\calO_K}\skakko*{1-\dfrac{\overline{\psi}(\frakp)}{N\frakp^s}}.
\end{align}
For a non-zero ideal $\frakg$ of $\calO_K$,
we define $L_{\frakg}(\overline{\psi}, s)$ in the same way. 
Fix $\overline{\bbQ}$ and $\overline{\bbQ_2}$ as algebraic closures of $\bbQ$ and $\bbQ_2$,
and fix embeddings $\overline{\bbQ} \hookrightarrow \overline{\bbQ_2}$
and $\overline{\bbQ} \hookrightarrow \bbC$.
Let $v_2$ denote the $2$-adic valuation of $\bbQ_2$ normalized so that $v_2(2)=1$
and extend to $\overline{\bbQ_2}$, which is also written as $v_2$.

\begin{prop}\label{prop:Tate_alg}
    Suppose $D\in \calO_K$ is congruent to $1$ modulo $2+2i$.
    The elliptic curve $E_{-D}$ has bad reduction at all primes dividing $D\calO_K$.
    Moreover, $E_{-D}$ has good reduction at the prime $(1+i)\calO_K$ if and only if $(i/D)_4=i$, where $(\cdot/\cdot)_4$ is the quartic residue character.
    In particular, we have
    \begin{align}
        v_2\skakko*{\dfrac{L_{2D}(\overline{\psi_{-D}}, 1)}{\Omega}}
        =v_2\skakko*{\dfrac{L_{2}(\overline{\psi_{-D}}, 1)}{\Omega}}=
        \begin{cases}
            v_2\skakko*{\dfrac{L(\overline{\psi_{-D}}, 1)}{\Omega}}-\dfrac{1}{2} & ((i/D)_4=i), \vspace{5pt}\\
            v_2\skakko*{\dfrac{L(\overline{\psi_{-D}}, 1)}{\Omega}} & (\text{otherwise}).
        \end{cases}
    \end{align}
\end{prop}
\begin{proof}
    Since the discriminant of the equation $y^2=x^3+Dx$ is $(1+i)^{12}D^3$ and $D$ is quartic-free,
    the elliptic curve $E_{-D}$ is minimal at all primes dividing $D\calO_K$.
    Therefore, the first claim follows.
    We show that $E_{-D}$ has good reduction at $(1+i)\calO_K$ when $(i/D)_4=i$ using Tate's algorithm.
    In the other cases, we can show similarly that $E_{-D}$ has bad reduction at $(1+i)\calO_K$.
    From now on, we follow Silverman's notation and steps \cite[p.366]{Silverman:1994}.
    
    We start from step 1. Set $\pi=1+i$ and we have
    \begin{gather}
        \Delta=\pi^{12}D^3, \quad a_1=a_2=a_3=a_6=0, \quad a_4=D,\\
        b_2=b_6=0, \quad b_4=2D, \quad b_8=-D^2.
    \end{gather}
    Since $\pi\mid \Delta,$ we proceed to Step 2.
    The curve $\tilde{E}$ obtained by reduction of $E$ at $\pi$ has the singular point $(1, 0)$.
    Thus, we do the transformation $x\mapsto x+1$ and obtain the new equation
    \begin{align}
        y^2=x^3+3x^2+(D+3)x+(D+1)
    \end{align}
    whose reduction curve has the singular point $(0, 0)$. Then, we have
    \begin{gather}
        a_1=a_3=0, \quad a_2=3, \quad a_4=D+3, \quad a_6=D+1,\\
        b_2=12, \quad b_4=2D+6, \quad b_6=4D+4, \quad b_8=-D^2+6D+3.
    \end{gather}
    We can easily check $\pi\mid b_2, \pi^2\mid a_6, \pi^3\mid b_6, b_8$ and proceed to Step 6.
    Let $k$ be the residue field $\calO_K/(\pi)$ and fix an algebraic closure $\bar{k}$.
    The following equations over $k$
    \begin{gather}
        Y^2+a_1Y-a_2\equiv (Y-\alpha)^2\bmod \pi,\\
        Y^2+a_{3, 1}Y-a_{6, 2}\equiv (Y-\beta)^2\bmod \pi
    \end{gather}
    have the solution $\alpha=\beta=1$.
    Thus, we do the transformation $y\mapsto y+x+\pi$ and obtain the new equation
    \begin{align}
        y^2+2xy+2\pi y=x^3+2x^2+(D+3-2\pi)x+(D+1-\pi^2).
    \end{align}
    Then, we have
    \begin{gather}
        a_1=a_2=2, \quad a_3=2\pi, \quad a_4=D+3-2\pi, \quad a_6=D+1-\pi^2,\\
        b_2=12, \quad b_4=2D+6, \quad b_6=4D+4, \quad b_8=-D^2+6D+3.
    \end{gather}
    We consider the polynomial over $k$
    \begin{align}
        P(T)=T^3+a_{2, 1}T^2+a_{4, 2}T+a_{6, 3}.
    \end{align}
    If we write $D=1+(2+2i)(s+ti)$ for $s, t\in \bbZ$, then we see that $P(T)=T^3-(s-t-1)$.
    By properties of the quartic residue symbol, $(i/D)_4=i$ is equivalent to $s-t\equiv 3\bmod 4$.
    Thus, $P(T)$ has the triple root $T=0$ and we proceed to Step 8.
    Since the polynomial over $k$
    \begin{align}
        Y^2+a_{3, 2}Y-a_{6, 4}=Y^2-s
    \end{align}
    has the double root $Y=0$ if $s\equiv 0\bmod 2$ and $Y=1$ if $s\equiv 1\bmod 2$.
    We suppose $s\equiv 0 \bmod 2$ and proceed to Step 9.
    (For the case $s\equiv 1\bmod 2$, we proceed to Step 9 after transformation $y\mapsto y+\pi^2$.)
    Since $\pi^4\mid a_4$ and $\pi^6\mid a_6$, we proceed to Step 11.
    Then, the transformation $x\mapsto \pi^2x, y\mapsto \pi^3y$ leads to the new equation
    \begin{align}
        y^2+\frac{2}{\pi}xy+\frac{2}{\pi^2}y=x^3+\frac{2}{\pi^2}x^2+\frac{D+3-2\pi}{\pi^4}x+\frac{D+1-\pi^2}{\pi^6}
    \end{align}
    whose discriminant is $D^3$. Therefore, the elliptic curve $E$ has good reduction at $(1+i)\calO_K$ and we finish Tate's algorithm.

    When $(i/D)_4=i$, $\psi_{-D}((1+i))$ is non-zero and equal to $u(1+i)$ for some $u\in \calO_K^\times$.
    Therefore, we have
    \begin{align}
        v_2\skakko*{\frac{\psi_{-D}((1+i))}{N(1+i)}}=v_2\skakko*{\frac{u(1+i)}{2}}=-\frac{1}{2}\neq 0.
    \end{align}
    Thus, the 2-adic valuation of the Euler factor at $(1+i)\calO_K$ is equal to $-1/2$.
\end{proof}

\begin{rem}
    If $(i/D)_4=i$, then the minimal model of $E_{-D}$ at $(1+i)\calO_K$ is
    \begin{align}
        \begin{cases}
            y^2+(1-i)xy-iy=x^3-ix^2-\frac{D+1-2i}{4}x+\frac{iD+2+i}{8} & (s\equiv 0\bmod 2),\\
            y^2+(1-i)xy+(1-2i)y=x^3-ix^2-\frac{D+1-6i}{4}x+\frac{iD+6+9i}{8} & (s\equiv 1\bmod 2),
        \end{cases}
    \end{align}
    where $D=1+(2+2i)(s+ti) \ (s, t\in \bbZ)$.
    Local informations at $(1+i)\calO_K$ including Kodaira symbols
    is summarized in TABLE \ref{tb:LocalInformation}.
    For other primes that divide $D$,
    we obtain TABLE \ref{tb:LocalInformation2} by Tate's algorithm.
\end{rem}

\begin{table}[htbp]
    \centering
    \begin{tabular}{c|ccccc}
        $(i/D)_4$ & \textrm{Kodaira Symbol} & $m$ & $v$ & $f$ & $c$ \\ \hline
        $\pm 1$ & $\textrm{I}_0^\ast$ & 5 & 12 & 8 & 2 \\
        $i$ & $\textrm{I}_0$ & 1 & 0 & 0 & 1 \\
        $-i$ & $\textrm{I}\hspace{-.1em}\textrm{I}^\ast$ & 9 & 12 & 4 & 1 \\
    \end{tabular}
    \caption{Local informations at $(1+i)\calO_K$}
    \label{tb:LocalInformation}
\end{table}

\begin{table}[htbp]
    \centering
    \begin{tabular}{c|ccccc}
        $\pi\mid D$ & \textrm{Kodaira Symbol} & $m$ & $v$ & $f$ & $c$ \\ \hline
        $\pi\mid D_1$ & $\textrm{I}\hspace{-.1em}\textrm{I}\hspace{-.1em}\textrm{I}$ & 2 & 3 & 2 & 2 \\
        $\pi\mid D_2, \ (D/\pi)_2=1$ & $\textrm{I}_0^\ast$ & 5 & 6 & 2 & 4 \\
        $\pi\mid D_2, \ (D/\pi)_2=-1$ & $\textrm{I}_0^\ast$ & 5 & 6 & 2 & 2 \\
        $\pi\mid D_3$ & $\textrm{I}\hspace{-.1em}\textrm{I}\hspace{-.1em}\textrm{I}^\ast$ & 8 & 9 & 2 & 2 \\
    \end{tabular}
    \caption{Local informations at $\pi \calO_K \ (\pi\mid D)$}
    \label{tb:LocalInformation2}
\end{table}

As mentioned in the last paragraph of Section \ref{sec:Introduction},
we iterate Zhao's method. For this purpose,
we first decompose $D$ uniquely up to units in $\calO_K$
according to the index of a prime dividing $D$,
such as $D_1^{(n)}D_2^{(m)}D_3^{(\ell)}$, where
\begin{align}
    D_1^{(n)}=\prod_{\pi_{1, i}\in S_1}\pi_{1, i}, \quad D_2^{(m)}=\prod_{\pi_{2, j}\in S_2}\pi_{2, j}^2, \quad D_3^{(\ell)}=\prod_{\pi_{3, k}\in S_3}\pi_{3, k}^3,
\end{align}
and $S_1=\{\pi_{1, 1}, \dots, \pi_{1, n}\}, S_2=\{\pi_{2, 1}, \dots, \pi_{2, m}\}, S_3=\{\pi_{3, 1}, \dots, \pi_{3, \ell}\}$
are disjoint sets of distinct primes of $\calO_K$ which are coprime to $2$.
Here, a prime of $\calO_K$ is said to be primary if it is congruent to $1$ modulo $2+2i$.
For a prime $\pi$ which is coprime to $2$,
it is known that only one of $\{\pm \pi, \pm i\pi\}$ is primary.
Therefore, all primes in $S_i \ (i=1, 2, 3)$ are assumed to be primary,
and $D$ is congruent to $1$ modulo $2+2i$.
We abbreviate $D_i^{(\ast)}$ as $D_i$
if we do not care about the number of the primes in $S_i$.

Next, we represent all divisors $D_T$ of $D$ as follows.
Let $T_1\subset \{1, \dots, n\}, T_2\subset \{1, \dots, m\}, T_3\subset \{1, \dots, \ell\}$ be arbitrary subsets
(including the case where $T_1, T_2$, and $T_3$ are empty sets).
Then, we define
\begin{align}
    D_{T_1}=\prod_{i\in T_1}\pi_{1, i}, \quad
    D_{T_2}=\prod_{j\in T_2}\pi_{2, j}^2, \quad
    D_{T_3}=\prod_{k\in T_3}\pi_{3, k}^3
\end{align}
and $D_T=D_{T_1}D_{T_2}D_{T_3}$.
When $T_i=\varnothing \ (i=1, 2, 3)$, we define $D_{T_i}=1$.

For a lattice $\calL$ of $\bbC$ and integer $k\geq 0$,
we define the holomorphic function on the domain $\Re(s)>1+k/2$ by
\begin{align}
    H_k(z, s, \calL)=\sum_{w\in \calL}{}^{'}\frac{\overline{(z+w)}^k}{|z+w|^{2s}}.
\end{align}
Here, $\sum'$ implies that $w=-z$ is excluded if $z\in \calL$.
The function $s\mapsto H_k(z, s, \calL)$ has the analytic continuation
to the entire complex $s$-plane if $k\geq 1$. We set
\begin{align}
    \calE_1^\ast(z, \calL)=H_1(z, 1, \calL).
\end{align}

\begin{prop}[{\cite[p.201, Proposition 5.5]{GoldsteinSchappacher:1981}}]\label{prop:GS}
    Let $E$ be an elliptic curve over an imaginary quadratic field $K$
    with complex multiplication by $\calO_K$.
    Fix a Weierstrass model of $E$ and take $\Omega_E\in \bbC^\times$
    such that the period lattice of $E$ is $\Omega_E\calO_K$.
    We write $\phi$ as the Hecke character of $K$ associated to $E$
    and suppose the conductor of $\phi$ divides a non-zero integral ideal $\frakg$ of $K$.
    Let $B$ be a minimal set consisting of ideals prime to $\frakg$ such that
    \begin{align}
        \Gal(K(E[\frakg])/K)=\mkakko*{\sigma_{\frakb}\mid \frakb\in B},
    \end{align}
    where $\sigma_{\frakb}$ is the Artin symbol corresponding to $\frakb$.
    We take $\rho\in \Omega_E K^\times$ such that $\rho\Omega_E^{-1}\calO_K=\frakg^{-1}$.
    Then, for $k\geq 1$, the following holds:
    \begin{align}
        \frac{\overline{\rho}^k}{|\rho|^{2s}}L_\frakg(\overline{\phi}^k, s)
        =\sum_{\frakb\in B}H_k(\phi(\frakb)\rho, s, \calL).
    \end{align}
\end{prop}

For the moment, we take $\Delta\in \calO_K\setminus\calO_K^\times$,
which is congruent to 1 modulo $2+2i$,
so that the conductor of $\psi_{-D_T}$ divides $4\Delta\calO_K$.
Later, we explicitly define $\Delta$ (see the paragraph after Lemma \ref{lem:Svaluation}).

\begin{lem}\label{lem:setB}
    We apply Proposition \ref{prop:GS} to
    $E=E_{-D_T}, \phi=\psi_{-D_T}, \frakg=4\Delta\calO_K$.
    Then a set $B$ can be taken as
    \begin{align}
        B=\mkakko*{(4c+\Delta), (4c+(1+2i)\Delta)\mid c\in \calC},
    \end{align}
    where $\calC$ is a complete system of representatives of $(\calO_K/\Delta \calO_K)^\times$.
\end{lem}
\begin{proof}
    Since the conductor of $\overline{\psi_{-D_T}}$ divides $4\Delta\calO_K$,
    from \cite[p.196, Lemma 4.7]{GoldsteinSchappacher:1981},
    we have $K(E_{-D_T}[4\Delta])=K(4\Delta)$.
    Thus the following isomorphism via the Artin map holds:
    \begin{align}
        \mathrm{Gal}(K(E_{-D_T}[4\Delta])/K)
        \simeq (\calO_K/4\Delta\calO_K)^\times/\calO_K^\times.\label{eq:setB}
    \end{align}
    Hence the cardinality of $B$ must be equal to
    $2 \cdot \# (\calO_K/\Delta\calO_K)^\times$.
    Therefore, it is sufficient to show that the Artin symbols
    corresponding to any two different elements in $B$ are different from each other.
    We show that $\sigma_{(4c+\Delta)}\neq \sigma_{(4c'+\Delta)}$ for $c\neq c'\in \calC$.
    Assume that $\sigma_{(4c+\Delta)}=\sigma_{(4c'+\Delta)}$.
    Then $4c+\Delta$ must be congruent to $4c'+\Delta$ modulo $4\Delta$.
    However, this implies that $c$ and $c'$ belong same equivalence class
    in $(\calO_K/\Delta\calO_K)^\times$, which is a contradiction.
    Other cases can be shown in the same way.
\end{proof}

We define the sign of $\Delta$ by $\sgn(\Delta)=1$ if $\Delta\equiv 1\bmod 4$
and $\sgn(\Delta)=-1$ if $\Delta \equiv 3+2i\bmod 4$.
For simplicity, we set
\begin{align}
    \varepsilon_T
    =\sgn(\Delta)\skakko*{\dfrac{-1}{D_T}}_4^{\frac{1+\sgn(\Delta)}{2}}\in \{\pm 1\}.
\end{align}

\begin{lem}\label{lem:Heckechar}
    For $c\in \calC$, we have
    \begin{align}
        &\psi_{-D_T}((4c+\Delta))=\varepsilon_T\overline{\skakko*{\dfrac{c}{D_T}}_4}(4c+\Delta),\\
        &\psi_{-D_T}((4c+(1+2i)\Delta))=\varepsilon_T\overline{\skakko*{\dfrac{c}{D_T}}_4}(4c+(1+2i)\Delta).
    \end{align}
\end{lem}
\begin{proof}
    As is well-known, for an ideal $\fraka$ of $\calO_K$ prime to $4D_T$, it holds that
    \begin{align}
        \psi_{-D_T}(\fraka)=\overline{\skakko*{\dfrac{-D_T}{\alpha}}_4}\alpha \quad (\fraka=(\alpha), \ \alpha\equiv 1\bmod 2+2i).
    \end{align}
    For example, see \cite[p.185, Exercise 2.34]{Silverman:1994}.
    Since $4c+\Delta\equiv 1\bmod 2+2i$, we have
    \begin{align}
        \psi_{-D_T}((4c+\Delta))
        &=\overline{\skakko*{\dfrac{-D_T}{4c+\Delta}}_4}(4c+\Delta)\\
        &=\skakko*{\dfrac{-1}{4c+\Delta}}_4\overline{\skakko*{\dfrac{D_T}{4c+\Delta}}_4}(4c+\Delta)\\
        &=\sgn(\Delta)\overline{\skakko*{\dfrac{D_T}{4c+\Delta}}_4}(4c+\Delta).
    \end{align}
    Let $p_{T_i}$ be the number of distinct primes that divide $D_{T_i}$
    and that are congruent to $3+2i$ modulo $4$.
    First, we consider the case of $\sgn(\Delta)=+1$.
    By the quartic reciprocity law, we can calculate as follows:
    \begin{align}
        \skakko*{\dfrac{D_T}{4c+\Delta}}_4
        &=\prod_{i\in T_1}\skakko*{\dfrac{\pi_{1, i}}{4c+\Delta}}_4\prod_{j\in T_2}\skakko*{\dfrac{\pi_{2, j}}{4c+\Delta}}_4^2\prod_{k\in T_3}\skakko*{\dfrac{\pi_{3, k}}{4c+\Delta}}_4^3\\
        &=\prod_{i\in T_1}\skakko*{\dfrac{4c+\Delta}{\pi_{1, i}}}_4\prod_{j\in T_2}\skakko*{\dfrac{4c+\Delta}{\pi_{2, j}}}_4^2\prod_{k\in T_3}\skakko*{\dfrac{4c+\Delta}{\pi_{3, k}}}_4^3\\
        &=\prod_{i\in T_1}\skakko*{\dfrac{-c}{\pi_{1, i}}}_4\prod_{j\in T_2}\skakko*{\dfrac{-c}{\pi_{2, j}}}_4^2\prod_{k\in T_3}\skakko*{\dfrac{-c}{\pi_{3, k}}}_4^3\\
        &=(-1)^{p_{T_1}+p_{T_3}}\prod_{i\in T_1}\skakko*{\dfrac{c}{\pi_{1, i}}}_4\prod_{j\in T_2}\skakko*{\dfrac{c}{\pi_{2, j}}}_4^2\prod_{k\in T_3}\skakko*{\dfrac{c}{\pi_{3, k}}}_4^3\\
        &=\skakko*{\dfrac{-1}{D_T}}_4\skakko*{\dfrac{c}{D_T}}_4.
    \end{align}
    In the same way, if $\sgn(\Delta)=-1$, then
    \begin{align}
        \skakko*{\dfrac{D_T}{4c+\Delta}}_4
        &=\prod_{i\in T_1}\skakko*{\dfrac{\pi_{1, i}}{4c+\Delta}}_4\prod_{j\in T_2}\skakko*{\dfrac{\pi_{2, j}}{4c+\Delta}}_4^2\prod_{k\in T_3}\skakko*{\dfrac{\pi_{3, k}}{4c+\Delta}}_4^3\\
        &=(-1)^{p_{T_1}+p_{T_3}}\prod_{i\in T_1}\skakko*{\dfrac{4c+\Delta}{\pi_{1, i}}}_4\prod_{j\in T_2}\skakko*{\dfrac{4c+\Delta}{\pi_{2, j}}}_4^2\prod_{k\in T_3}\skakko*{\dfrac{4c+\Delta}{\pi_{3, k}}}_4^3\\
        &=(-1)^{p_{T_1}+p_{T_3}}\prod_{i\in T_1}\skakko*{\dfrac{-c}{\pi_{1, i}}}_4\prod_{j\in T_2}\skakko*{\dfrac{-c}{\pi_{2, j}}}_4^2\prod_{k\in T_3}\skakko*{\dfrac{-c}{\pi_{3, k}}}_4^3\\
        &=\prod_{i\in T_1}\skakko*{\dfrac{c}{\pi_{1, i}}}_4\prod_{j\in T_2}\skakko*{\dfrac{c}{\pi_{2, j}}}_4^2\prod_{k\in T_3}\skakko*{\dfrac{c}{\pi_{3, k}}}_4^3\\
        &=\skakko*{\dfrac{c}{D_T}}_4.
    \end{align}
    The rest can be proved similarly.
\end{proof}

\begin{lem}\label{lem:Eseries}
    Denote the period lattice $\Omega\calO_K$ of $E_1: y^2=x^3-x$ as $\calL_{\Omega}$.
    Let $\wp(z)=\wp(z, \calL_{\Omega})$ be the Weierstrass $\wp$-function
    and let $\zeta(z)=\zeta(z, \calL_{\Omega})$ be the Weierstrass $\zeta$-function.
    Then for $c\in \calC$, we have
    \begin{align}
        &\calE_1^\ast\skakko*{\dfrac{c\Omega}{\Delta}+\dfrac{\Omega}{4}, \calL_{\Omega}}+\calE_1^\ast\skakko*{\dfrac{c\Omega}{\Delta}+\dfrac{(1+2i)\Omega}{4}, \calL_{\Omega}}\\
        &\begin{multlined}=2\mkakko*{\zeta\skakko*{\dfrac{c\Omega}{\Delta}}-\dfrac{\pi}{\Omega}\overline{\skakko*{\dfrac{c}{\Delta}}}}+\dfrac{\wp'(c\Omega/\Delta)}{2}\mkakko*{\dfrac{1}{\wp(c\Omega/\Delta)-(1+\sqrt{2})}+\dfrac{1}{\wp(c\Omega/\Delta)-(1-\sqrt{2})}}\\
        +\sqrt{2}+\mkakko*{\dfrac{2+\sqrt{2}}{\wp(c\Omega/\Delta)-(1+\sqrt{2})}-\dfrac{2-\sqrt{2}}{\wp(c\Omega/\Delta)-(1-\sqrt{2})}}.
        \end{multlined}
    \end{align}
\end{lem}
\begin{proof}
    For a lattice $\calL=u\bbZ+v\bbZ \ (\Im(v/u)>0)$ of $\bbC$, we set
    \begin{align}
        s_2(\calL)=\lim_{s\to +0}\sum_{w\in \calL\setminus\{0\}}\dfrac{1}{w^2|w|^{2s}}, \quad A(\calL)=\dfrac{\bar{u}v-u\bar{v}}{2\pi i}.
    \end{align}
    Then, the identity $\calE_1^\ast(z, \calL)=\zeta(z, \calL)-zs_2(\calL)-\bar{z}A(\calL)^{-1}$
    holds (\textrm{cf}. \cite[p.188, Proposition 1.5]{GoldsteinSchappacher:1981}).
    It is easy to see  $s_2(\calL_{\Omega})=0$ and $A(\calL_{\Omega})=\Omega^2/\pi$.
    Hence, we see that
    \begin{align}
        \calE_1^\ast(z, \calL_{\Omega})=\zeta(z, \calL_{\Omega})-\dfrac{\pi\bar{z}}{\Omega^2}.\label{eq:Eseries1}
    \end{align}
    The addition formula
    \begin{align}
        \zeta(z_1+z_2, \calL)=\zeta(z_1, \calL)+\zeta(z_2, \calL)+\dfrac{1}{2}\dfrac{\wp'(z_1, \calL)-\wp'(z_2, \calL)}{\wp(z_1, \calL)-\wp(z_2, \calL)}
    \end{align}
    and equation \eqref{eq:Eseries1} lead to 
    \begin{align}
        \calE_1^\ast\skakko*{\dfrac{c\Omega}{\Delta}+\dfrac{\Omega}{4}, \calL_{\Omega}}&=\zeta\skakko*{\dfrac{c\Omega}{\Delta}+\dfrac{\Omega}{4}}-\dfrac{\pi}{\Omega^2}\overline{\skakko*{\dfrac{c\Omega}{\Delta}+\dfrac{\Omega}{4}}}\\
        &=\zeta\skakko*{\dfrac{c\Omega}{\Delta}}+\zeta\skakko*{\dfrac{\Omega}{4}}+\dfrac{1}{2}\dfrac{\wp'(c\Omega/\Delta)-\wp'(\Omega/4)}{\wp(c\Omega/\Delta)-\wp(\Omega/4)}-\dfrac{\pi}{4\Omega}-\dfrac{\pi}{\Omega}\overline{\skakko*{\dfrac{c}{\Delta}}}.
    \end{align}
    Similarly, we obtain
    \begin{align}
        \calE_1^\ast&\skakko*{\dfrac{c\Omega}{\Delta}+\dfrac{(1+2i)\Omega}{4}, \calL_{\Omega}}\\
        &=\zeta\skakko*{\dfrac{c\Omega}{\Delta}}+\zeta\skakko*{\dfrac{(1+2i)\Omega}{4}}+\dfrac{1}{2}\dfrac{\wp'(c\Omega/\Delta)-\wp'((1+2i)\Omega/4)}{\wp(c\Omega/\Delta)-\wp((1+2i)\Omega/4)}-\dfrac{(1-2i)\pi}{4\Omega}-\dfrac{\pi}{\Omega}\overline{\skakko*{\dfrac{c}{\Delta}}}.
    \end{align}
    Moreover from \cite[(2.7)]{Zhao:2003}, we know
    $\wp(\Omega/4)=1+\sqrt{2}, \wp'(\Omega/4)=-4-2\sqrt{2}, \wp((1+2i)\Omega/4)=1-\sqrt{2}, \wp'((1+2i)\Omega/4)=4-2\sqrt{2}$ and
    \begin{align}
        \zeta\skakko*{\dfrac{\Omega}{4}}+\zeta\skakko*{\dfrac{(1+2i)\Omega}{4}}-\dfrac{(1-i)\pi}{2\Omega}=\sqrt{2}.
    \end{align}
    By combining these results, the lemma holds.
\end{proof}

\begin{thm}[cf. \cite{BirchSwinnerton-Dyer:1965}]\label{thm:finite_sum}
    We put $\chi=\chi(D_T)=((1+i)/D_T)_4$. Then, the following holds:
    \begin{align}
        &\dfrac{\varepsilon_{T}\Delta}{\Omega}L_{2\Delta}(\overline{\psi_{-D_T}}, 1)\\
        =&
        \begin{cases}
            \dfrac{\sqrt{2}}{4}\displaystyle\sum_{c\in \calC}\skakko*{\dfrac{c}{D_T}}_4+\dfrac{1}{\sqrt{2}}\displaystyle\sum_{c\in\calC}\skakko*{\dfrac{c}{D_T}}_4\dfrac{\wp(c\Omega/\Delta)+1}{\wp(c\Omega/\Delta)^2-2\wp(c\Omega/\Delta)-1} & ((i/D_T)_4=\pm 1),\vspace{5pt}\\
            \dfrac{1}{8}\displaystyle\sum_{c\in \calC}\skakko*{\dfrac{c}{D_T}}_4\mkakko*{\dfrac{(1-i)\chi}{1-(1-i)\chi}\dfrac{\wp'(c\Omega/\Delta)}{\wp(c\Omega/\Delta)}+\dfrac{2\wp'(c\Omega/\Delta)(\wp(c\Omega/\Delta)-1)}{\wp(c\Omega/\Delta)^2-2\wp(c\Omega/\Delta)-1}} & ((i/D_T)_4=i),\vspace{5pt}\\
            \dfrac{1}{4}\displaystyle\sum_{c\in \calC}\skakko*{\dfrac{c}{D_T}}_4\dfrac{\wp'(c\Omega/\Delta)(\wp(c\Omega/\Delta)-1)}{\wp(c\Omega/\Delta)^2-2\wp(c\Omega/\Delta)-1} & ((i/D_T)_4=-i).
        \end{cases}
    \end{align}
\end{thm}
\begin{proof}
    Take $\Omega_{T}\in \bbC^\times$ so that the period lattice of the elliptic curve
    $E_{-D_T}: y^2=x^3+D_{T}x$ is $\Omega_{T} \calO_K$ and set $\alpha=\Omega/\Omega_{T}$.
    In Proposition \ref{prop:GS}, substituting $k=s=1, \frakg=(4\Delta), \rho=\Omega_{T}/(4\Delta), \calL=\Omega_{T}\calO_K$ leads to
    \begin{align}
        \dfrac{4\Delta}{\Omega_{T}}L_{2\Delta}(\overline{\psi_{-D_T}}, 1)
        =\sum_{\frakb\in B}\calE_1^\ast\skakko*{\psi_{-D_T}(\frakb)\dfrac{\Omega_{T}}{4\Delta}, \Omega_{T}\calO_K}\label{eq:finite_sum1}.
    \end{align}
    Moreover, by using Lemma \ref{lem:setB} and Lemma \ref{lem:Heckechar},
    the right-hand side of the equation \eqref{eq:finite_sum1} can be calculated as
    \begin{align}
        \sum_{c\in \calC}\calE_1^\ast\skakko*{\varepsilon_T\overline{\skakko*{\dfrac{c}{D_T}}_4}\dfrac{4c+\Delta}{4\Delta}\dfrac{\Omega}{\alpha}, \dfrac{\Omega}{\alpha}\calO_K}+\sum_{c\in \calC}\calE_1^\ast\skakko*{\varepsilon_T\overline{\skakko*{\dfrac{c}{D_T}}_4}\dfrac{4c+(1+2i)\Delta}{4\Delta}\dfrac{\Omega}{\alpha}, \dfrac{\Omega}{\alpha}\calO_K}.
    \end{align}
    Note that for $\lambda\in \bbC^\times$ and a lattice $\calL$ of $\bbC$,
    $\calE_1^\ast(\lambda z, \lambda \calL)=\lambda^{-1}\calE_1^\ast(z, \calL)$ holds.
    Thus, by Lemma \ref{lem:Eseries}, we have
    \begin{align}
        \dfrac{\varepsilon_T\Delta}{\Omega}L_{2\Delta}(\overline{\psi_{-D_T}}, 1)
        &=\dfrac{1}{4}\sum_{c\in \calC}\skakko*{\dfrac{c}{D_T}}_4\mkakko*{\calE_1^\ast\skakko*{\dfrac{c\Omega}{\Delta}+\dfrac{\Omega}{4}, \calL_{\Omega}}+\calE_1^\ast\skakko*{\dfrac{c\Omega}{\Delta}+\dfrac{(1+2i)\Omega}{4}, \calL_{\Omega}}}\\
        &=\dfrac{1}{4}\sum_{c\in \calC}\skakko*{\dfrac{c}{D_T}}_4(f_1(c)+f_2(c)+g(c)),\label{eq:finite_sum2}
    \end{align}
    where 
    \begin{align}
        &f_1(c)=2\mkakko*{\zeta\skakko*{\dfrac{c\Omega}{\Delta}}-\dfrac{\pi}{\Omega}\overline{\skakko*{\dfrac{c}{\Delta}}}},\\
        &f_2(c)=\dfrac{\wp'(c\Omega/\Delta)}{2}\mkakko*{\dfrac{1}{\wp(c\Omega/\Delta)-(1+\sqrt{2})}+\dfrac{1}{\wp(c\Omega/\Delta)-(1-\sqrt{2})}},\\
        &g(c)=\sqrt{2}+\mkakko*{\dfrac{2+\sqrt{2}}{\wp(c\Omega/\Delta)-(1+\sqrt{2})}-\dfrac{2-\sqrt{2}}{\wp(c\Omega/\Delta)-(1-\sqrt{2})}}.
    \end{align}
    The functions $f_1(c)$ and $f_2(c)$ are odd with respect to $c$, and $g(c)$ is even with respect to $c$.
    We prove by cases according to the value $(i/D_T)_4$.

    First, we consider the case of $(i/D_T)_4=\pm 1$. Since $(-1/D_T)_4=1$,
    the function $(c/D_T)_4$ is even with respect to $c$.
    We can take $\calC$ so that if $c\in \calC$, then $-c\in \calC$ because of $(2, \Delta)=1$.
    Thus $\sum_{c}(c/D_T)_4f_1(c)$ and $\sum_{c}(c/D_T)_4f_2(c)$ must be equal to $0$.
    Next, we consider the case of $(i/D_T)_4=-i$. Since $(-1/D_T)_4=-1$,
    the function $(c/D_T)_4$ is odd with respect to $c$.
    As in the previous case, $\sum_{c}(c/D_T)_4g(c)$ is equal to $0$.
    Furthermore, we can take $\calC$ so that if $c\in \calC$, then $ic\in \calC$.
    Then, the value
    \begin{align}
        &\skakko*{\dfrac{c}{D_T}}_4\mkakko*{\zeta\skakko*{\dfrac{c\Omega}{\Delta}}-\dfrac{\pi}{\Omega}\overline{\skakko*{\dfrac{c}{\Delta}}}}+\skakko*{\dfrac{ic}{D_T}}_4\mkakko*{\zeta\skakko*{\dfrac{ic\Omega}{\Delta}}-\dfrac{\pi}{\Omega}\overline{\skakko*{\dfrac{ic}{\Delta}}}}
    \end{align}
    is equal to $0$. Hence, we have $\sum_{c}(c/D_T)_4f_1(c)=0$.
    Finally, we consider the case of $(i/D_T)_4=i$. Note that the value
    \begin{align}
        \sum_{c\in \calC}\skakko*{\dfrac{c}{D_T}}_4\mkakko*{\zeta\skakko*{\dfrac{c\Omega}{\Delta}}-\dfrac{\pi}{\Omega}\overline{\skakko*{\dfrac{c}{\Delta}}}}
    \end{align}
    does not depend on the choice of $\calC$.
    In fact, we can show it by using the identities $\zeta(z+1)=\zeta(z)+\pi$ and $\zeta(z+i)=\zeta(z)-\pi i$.
    Therefore, the transformation $c\mapsto (1+i)c$ leads to
    \begin{align}
        &\sum_{c\in \calC}\skakko*{\dfrac{c}{D_T}}_4\mkakko*{\zeta\skakko*{\dfrac{c\Omega}{\Delta}}-\dfrac{\pi}{\Omega}\overline{\skakko*{\dfrac{c}{\Delta}}}}\\
        &=\sum_{c\in \calC}\skakko*{\dfrac{(1+i)c}{D_T}}_4\mkakko*{\zeta\skakko*{\dfrac{(1+i)c\Omega}{\Delta}}-\dfrac{\pi}{\Omega}\overline{\skakko*{\dfrac{(1+i)c}{\Delta}}}}\\
        &=\chi\sum_{c\in \calC}\skakko*{\dfrac{c}{D_T}}_4\mkakko*{\zeta\skakko*{\dfrac{c\Omega}{\Delta}}+\zeta\skakko*{\dfrac{ic\Omega}{\Delta}}+\dfrac{1}{2}\dfrac{\wp'(c\Omega/\Delta)-\wp'(ic\Omega/\Delta)}{\wp(c\Omega/\Delta)-\wp(ic\Omega/\Delta)}-\dfrac{(1-i)\pi}{\Omega}\overline{\skakko*{\dfrac{c}{\Delta}}}}\\
        &=(1-i)\chi \sum_{c\in \calC}\skakko*{\dfrac{c}{D_T}}_4\mkakko*{\zeta\skakko*{\dfrac{c\Omega}{\Delta}}+\dfrac{1}{4}\dfrac{\wp'(c\Omega/\Delta)}{\wp(c\Omega/\Delta)}-\dfrac{\pi}{\Omega}\overline{\skakko*{\dfrac{c}{\Delta}}}}\\
        &=(1-i)\chi\sum_{c\in \calC}\skakko*{\dfrac{c}{D_T}}_4\mkakko*{\zeta\skakko*{\dfrac{c\Omega}{\Delta}}-\dfrac{\pi}{\Omega}\overline{\skakko*{\dfrac{c}{\Delta}}}}+\dfrac{(1-i)\chi}{4}\sum_{c\in \calC}\skakko*{\dfrac{c}{D_T}}_4\dfrac{\wp'(c\Omega/\Delta)}{\wp(c\Omega/\Delta)}.
    \end{align}
    Thus, we see that
    \begin{align}
        \sum_{c\in \calC}\skakko*{\dfrac{c}{D_T}}_4\mkakko*{\zeta\skakko*{\dfrac{c\Omega}{\Delta}}-\dfrac{\pi}{\Omega}\overline{\skakko*{\dfrac{c}{\Delta}}}}=\dfrac{(1-i)\chi}{1-(1-i)\chi}\dfrac{1}{4}\sum_{c\in \calC}\skakko*{\dfrac{c}{D_T}}_4\dfrac{\wp'(c\Omega/\Delta)}{\wp(c\Omega/\Delta)}.
    \end{align}
    We substitute these results into \eqref{eq:finite_sum2} and the theorem follows.
\end{proof}

We set $\calP(c)=\wp(c\Omega/\Delta), \calP'(c)=\wp'(c\Omega/\Delta)$
and $L_{2\Delta}^{\ast}(\overline{\psi_{-D_T}}, 1)=\varepsilon_T \Delta L_{2\Delta}(\overline{\psi_{-D_T}}, 1)$ for simplicity.
Note that we have
\begin{align}
    v_2\skakko*{\dfrac{L_{2\Delta}^{\ast}(\overline{\psi_{-D_T}}, 1)}{\Omega}}
    =v_2\skakko*{\dfrac{L_{2\Delta}(\overline{\psi_{-D_T}}, 1)}{\Omega}}.
\end{align}
As in the proof of Theorem \ref{thm:finite_sum}, we take $\calC$ so that if $c\in \calC$, then $-c, \pm ic\in \calC$.
Let $V$ be the subset of $\calC$ consisting of all primary elements, that is,
\begin{align}
    V=\mkakko*{c\in \calC\mid c\equiv 1\bmod 2+2i}.
\end{align}
We can rewrite the sums over $\calC$ in Theorem \ref{thm:finite_sum} as the sums over $V$.
For example if $(i/D_T)_4=1$, then we have
\begin{align}
    \dfrac{1}{\sqrt{2}}\sum_{c\in \calC}\skakko*{\dfrac{c}{D_T}}_4\dfrac{\calP(c)+1}{\calP(c)^2-2\calP(c)-1}=\sum_{c\in V}\skakko*{\dfrac{c}{D_T}}_4\dfrac{2\sqrt{2}(3\calP(c)^2-1)}{(\calP(c)^2-2\calP(c)-1)(\calP(c)^2+2\calP(c)-1)}.
\end{align}
The same calculation yields the following corollary.

\begin{cor}\label{cor:finite_sum_cor}
    Under the same conditions as Theorem \ref{thm:finite_sum}, we have
    \begin{align}
        &\dfrac{L_{2\Delta}^{\ast}(\overline{\psi_{-D_T}}, 1)}{\Omega}\\
        =&
        \begin{cases}
            \dfrac{\sqrt{2}}{4}\displaystyle\sum_{c\in \calC}\skakko*{\dfrac{c}{D_T}}_4+\displaystyle\sum_{c\in V}\skakko*{\dfrac{c}{D_T}}_4\dfrac{2\sqrt{2}(3\calP(c)^2-1)}{(\calP(c)^2-2\calP(c)-1)(\calP(c)^2+2\calP(c)-1)} & ((i/D_T)_4=1),\\
            \dfrac{\sqrt{2}}{4}\displaystyle\sum_{c\in \calC}\skakko*{\dfrac{c}{D_T}}_4+\displaystyle\sum_{c\in V}\skakko*{\dfrac{c}{D_T}}_4\dfrac{2\sqrt{2}\calP(c)(\calP(c)^2+1)}{(\calP(c)^2-2\calP(c)-1)(\calP(c)^2+2\calP(c)-1)} & ((i/D_T)_4=-1),\\
            \displaystyle\sum_{c\in V}\skakko*{\dfrac{c}{D_T}}_4\dfrac{\calP'(c)}{\calP(c)}\dfrac{\chi(\calP(c)^4-6\calP(c)^2+1)+(\calP(c)^3+\calP(c))}{(\calP(c)^2-2\calP(c)-1)(\calP(c)+2\calP(c)-1)} & ((i/D_T)_4=i),\\
            \displaystyle\sum_{c\in V}\skakko*{\dfrac{c}{D_T}}_4\dfrac{\calP'(c)(\calP(c)^2+1)}{(\calP(c)^2-2\calP(c)-1)(\calP(c)^2+2\calP(c)-1)} & ((i/D_T)_4=-i).
        \end{cases}
    \end{align}
\end{cor}


\section{2-adic valuation}\label{sec:2AdicValuation}

In Corollary \ref{cor:finite_sum_cor}, we define
\begin{align}
    &W_1(c)=\dfrac{2\sqrt{2}(3\calP(c)^2-1)}{(\calP(c)^2-2\calP(c)-1)(\calP(c)^2+2\calP(c)-1)},\\
    &W_{-1}(c)=\dfrac{2\sqrt{2}\calP(c)(\calP(c)^2+1)}{(\calP(c)^2-2\calP(c)-1)(\calP(c)^2+2\calP(c)-1)},\\
    &W_{i}(c)=\dfrac{\calP'(c)}{\calP(c)}\dfrac{\chi(\calP(c)^4-6\calP(c)^2+1)+(\calP(c)^3+\calP(c))}{(\calP(c)^2-2\calP(c)-1)(\calP(c)+2\calP(c)-1)},\\
    &W_{-i}(c)=\dfrac{\calP'(c)(\calP(c)^2+1)}{(\calP(c)^2-2\calP(c)-1)(\calP(c)^2+2\calP(c)-1)}.
\end{align}

\begin{lem}\label{lem:Svaluation}
    For $c\in V$, it holds that
    \begin{align}
        v_2(W_{1}(c))=v_2(W_{-1}(c))=v_2(W_{i}(c))=v_2(W_{-i}(c))=-\dfrac{1}{2}.
    \end{align}
\end{lem}
\begin{proof}
    \cite[Lemma 5, p.90]{BirchSwinnerton-Dyer:1965} shows
    \begin{gather}
        v_2(\calP(c)^2-2\calP(c)-1)=v_2(\calP(c)^2+2\calP(c)-1)=\dfrac{7}{4},\\
        v_2(\calP(c)-1)=\dfrac{1}{2}, \quad v_2(\calP(c)^2-1)=1, \quad v_2(\calP(c)^2+1)=\dfrac{3}{2}.
    \end{gather}
    Thus, we have
    \begin{align}
        v_2(3\calP(c)^2-1)=v_2(\calP(c)^2-3)=\dfrac{3}{2}, \quad v_2(\calP(c))=0.
    \end{align}
    By the identity $\wp'(z)^2=4\wp(z)^3-4\wp(z)$, we also have $v_2(\calP'(c))=3/2$. The claim follows from here.
\end{proof}

We consider a summation over $T$ of the equations in Corollary \ref{cor:finite_sum_cor}.
Here, we divide the range of $T$ into several cases
and define $\Delta$ for each of these cases.
We define $\Delta_i$ as the radical of $D_i$; that is;
\begin{align}
    \Delta_1=\prod_{\pi_{1, i}\in S_1}\pi_{1, i}, \quad
    \Delta_2=\prod_{\pi_{2, j}\in S_2}\pi_{2, j}, \quad
    \Delta_3=\prod_{\pi_{3, k}\in S_3}\pi_{3, k}.
\end{align}
If we run $T_1$ over all subsets of $\{1, \dots, n\}$ and $T_2=T_3=\varnothing$,
then we define $\Delta$ to be $\Delta_1$. In general, we define $\Delta$ as follows:
\begin{align}
    \Delta=
    \begin{cases}
        \Delta_1 & (T_1\subset \{1, \dots, n\}, T_2=\varnothing, T_3=\varnothing),\\
        \Delta_2 & (T_1= \varnothing, T_2\subset \{1, \dots, m\}, T_3=\varnothing),\\
        \Delta_3 & (T_1= \varnothing, T_2=\varnothing, T_3\subset \{1, \dots, \ell\}),\\
        \Delta_1\Delta_2 & (T_1\subset \{1, \dots, n\}, T_2\subset \{1, \dots, m\}, T_3= \varnothing),\\
        \Delta_2\Delta_3 & (T_1= \varnothing, T_2\subset\{1, \dots, m\}, T_3\subset\{1, \dots, \ell\}),\\
        \Delta_1\Delta_3 & (T_1\subset \{1, \dots, n\}, T_2=\varnothing, T_3\subset\{1, \dots, \ell\}),\\
        \Delta_1\Delta_2\Delta_3 & (T_1\subset \{1, \dots, n\}, T_2\subset \{1, \dots, m\}, T_3\subset \{1, \dots, \ell\}).
    \end{cases}
\end{align}
Here, for example, ``$T_1\subset \{1, \dots, n\}, T_2=\varnothing, T_3=\varnothing$'' implies
``$T_1$ runs over all subsets of $\{1, \dots, n\}$, $T_2$ and $T_3$ are both empty''.
From \cite[p.514, Theorem 12]{Serre:1968}, we see that
the conductor of the elliptic curve $E_{-D_T}$
is the square of the conductor of the Hecke character $\psi_{-D_T}$.
Therefore by TABLE \ref{tb:LocalInformation} and TABLE \ref{tb:LocalInformation2},
the conductor of $\psi_{-D_T}$ divides $4\Delta\calO_K$.

\begin{lem}\label{lem:charvaluation}
    We have the following lower bound of the $2$-adic valuation:
    \begin{align}
        v_2\skakko*{\sum_T\skakko*{\dfrac{c}{D_T}}_4}\geq
        \begin{cases}
            n/2 & (T_1\subset \{1, \dots, n\}, T_2=\varnothing, T_3=\varnothing),\\
            m & (T_1=\varnothing, T_2\subset \{1, \dots, m\}, T_3=\varnothing),\\
            \ell/2 & (T_1=\varnothing, T_2=\varnothing, T_3\subset \{1, \dots, \ell\}),\\
            n/2+m & (T_1\subset \{1, \dots, n\}, T_2\subset \{1, \dots, m\}, T_3=\varnothing), \\
            m+\ell/2 & (T_1=\varnothing, T_2\subset \{1, \dots, m\}, T_3\subset \{1, \dots, \ell\}), \\
            (n+\ell)/2 & (T_1\subset \{1, \dots, n\}, T_2=\varnothing, T_3\subset \{1, \dots, \ell\}), \\
            n/2+m+\ell/2 & (T_1\subset \{1, \dots, n\}, T_2\subset \{1, \dots, m\}, T_3\subset \{1, \dots, \ell\}).
        \end{cases}
    \end{align}
\end{lem}
\begin{proof}
    We only prove the case of $D_T=D_{T_1}$. We only need to show that 
    \begin{align}
        \sum_{T_1\subset \{1, \dots, n\}}\skakko*{\dfrac{c}{D_{T_1}}}_4=\mkakko*{1+\skakko*{\dfrac{c}{\pi_{1, 1}}}_4}\cdots \mkakko*{1+\skakko*{\dfrac{c}{\pi_{1, n}}}_4}.
    \end{align}
    We show by induction on $n$. Clearly, it holds for $n=1$.
    Suppose it is true for $1, \dots, n-1$. Then, we have
    \begin{align}
        \sum_{T_1\subset \{1, \dots, n\}}\skakko*{\dfrac{c}{D_{T_1}}}_4
        &=\sum_{T_1\subset \{1, \dots, n-1\}}\skakko*{\dfrac{c}{D_{T_1}}}_4+\sum_{\substack{T_1\subset \{1, \dots, n\} \\ n\in T_1}}\skakko*{\dfrac{c}{D_{T_1}}}_4\\
        &=\sum_{T_1\subset \{1, \dots, n-1\}}\skakko*{\dfrac{c}{D_{T_1}}}_4+\skakko*{\dfrac{c}{\pi_{1, n}}}_4\sum_{T_1\subset \{1, \dots, n-1\}}\skakko*{\dfrac{c}{D_{T_1}}}_4\\
        &=\mkakko*{1+\skakko*{\dfrac{c}{\pi_{1, n}}}_4}\sum_{T_1\subset \{1, \dots, n-1\}}\skakko*{\dfrac{c}{D_{T_1}}}_4\\
        &=\mkakko*{1+\skakko*{\dfrac{c}{\pi_{1, 1}}}_4}\cdots \mkakko*{1+\skakko*{\dfrac{c}{\pi_{1, n}}}_4},
    \end{align}
where the last equality follows from the induction hypothesis.
Thus, it is true for $n$. This completes the proof. 
\end{proof}

\begin{prop}\label{prop:finitesumvaluation}
The following holds:
    \begin{align}
        &v_2\skakko*{\sum_{T}\dfrac{L_{2\Delta}^{\ast}(\overline{\psi_{-D_T}}, 1)}{\Omega}}\\
        &\geq
        \begin{cases}
            \dfrac{n-1}{2} & (T_1\subset \{1, \dots, n\}, T_2=\varnothing, T_3=\varnothing),\vspace{2pt}\\
            \dfrac{2m-1}{2} & (T_1=\varnothing, T_2\subset \{1, \dots, m\}, T_3=\varnothing),\vspace{2pt}\\
            \dfrac{\ell-1}{2} & (T_1=\varnothing, T_2=\varnothing, T_3\subset \{1, \dots, \ell\}),\vspace{2pt}\\
            \dfrac{n+2m-1}{2} & (T_1\subset \{1, \dots, n\}, T_2\subset \{1, \dots, m\}, T_3=\varnothing),\vspace{2pt}\\
            \dfrac{2m+\ell-1}{2} & (T_1=\varnothing, T_2\subset \{1, \dots, m\}, T_3\subset \{1, \dots, \ell\}),\vspace{2pt}\\
            \dfrac{n+\ell-1}{2} & (T_1\subset \{1, \dots, n\}, T_2=\varnothing, T_3\subset \{1, \dots, \ell\}),\vspace{2pt}\\
            \dfrac{n+2m+\ell-1}{2} & (T_1\subset \{1, \dots, n\}, T_2\subset \{1, \dots, m\}, T_3\subset \{1, \dots, \ell\}).
        \end{cases}
    \end{align}
\end{prop}
\begin{proof}
    We only prove the case of $D_T=D_{T_1}$.
    Consider the summation over $T_1$ for the equations in Corollary \ref{cor:finite_sum_cor}.
    Then, we have $\Delta=\Delta_1$ and
    \begin{align}
        \sum_{T_1}\dfrac{L_{2\Delta_1}^{\ast}(\overline{\psi_{-D_{T_1}}}, 1)}{\Omega}&=
        \begin{cases}
            \displaystyle\dfrac{\sqrt{2}}{4}\sum_{T_1}\sum_{c\in \calC}\skakko*{\dfrac{c}{D_{T_1}}}_4+\displaystyle\sum_{c\in V}W_1(c)\sum_{T_1}\skakko*{\dfrac{c}{D_{T_1}}}_4 & ((i/D_{T_1})_4=1),\\
            \displaystyle\dfrac{\sqrt{2}}{4}\sum_{T_1}\sum_{c\in \calC}\skakko*{\dfrac{c}{D_{T_1}}}_4+\displaystyle\sum_{c\in V}W_{-1}(c)\sum_{T_1}\skakko*{\dfrac{c}{D_{T_1}}}_4 & ((i/D_{T_1})_4=-1),\\
            \displaystyle\sum_{c\in V}W_{i}(c)\sum_{T_1}\skakko*{\dfrac{c}{D_{T_1}}}_4 & ((i/D_{T_1})_{4}=i),\\
            \displaystyle\sum_{c\in V}W_{-i}(c)\sum_{T_1}\skakko*{\dfrac{c}{D_{T_1}}}_4 & ((i/D_{T_1})_{4}=-i),
        \end{cases}
    \end{align}
    where $T_1$ runs over all subsets of $\{1, \dots, n\}$.
    By Lemma \ref{lem:Svaluation} and Lemma \ref{lem:charvaluation},
    for any $\circ\in \{\pm 1, \pm i\}$, we have
    \begin{align}
        v_2\skakko*{\sum_{c\in V}W_{\circ}(c)\sum_{T_1}\skakko*{\dfrac{c}{D_{T_1}}}_4}&\geq \min_{c\in V}\mkakko*{v_2\skakko*{W_{\circ}(c)}+ v_2\skakko*{\sum_{T_1}\skakko*{\dfrac{c}{D_{T_1}}}_4}}\\
        &=-\dfrac{1}{2}+\dfrac{n}{2}.
    \end{align}
    Since 
    \begin{align}
        \sum_{c\in \calC}\skakko*{\dfrac{c}{D_{T_1}}}_4=
        \begin{cases}
        0 & (T_1\neq \varnothing),\\
        \# \calC & (T_1=\varnothing),
        \end{cases}
    \end{align}
    it holds that
    \begin{align}
        v_2\skakko*{\dfrac{\sqrt{2}}{4}\sum_{T_1}\sum_{c\in \calC}\skakko*{\dfrac{c}{D_{T_1}}}_4}=v_2(\# \calC)-\dfrac{3}{2}\geq 2n-\dfrac{3}{2}>\dfrac{n-1}{2}.
    \end{align}
    The proposition follows from this.
\end{proof}

\begin{thm}\label{thm:Lvaluation}
    Let $\psi_{-D}$ be the Hecke character associated to the elliptic curve
    $E_{-D}: y^2=x^3+Dx$ over $K$ and
    $\Omega$ is the least positive element of the period lattice of $E_1: y^2=x^3-x$.
    Then, we have
    \begin{align}
    v_2\skakko*{\dfrac{L_{2}(\overline{\psi_{-D}}, 1)}{\Omega}}\geq
    \begin{cases}
        \dfrac{n-2}{2} & (D=D_{1}^{(n)}),\vspace{2pt}\\
        \dfrac{2m-3}{2} & (D=D_{2}^{(m)}),\vspace{2pt}\\
        \dfrac{\ell-2}{2} & (D=D_{3}^{(\ell)}).
    \end{cases}
    \end{align}
\end{thm}
\begin{proof}
    We only prove the case of $D=D_{1}^{(n)}$.
    When $T_1=\{1, \dots, n\}$, we see that $E_{-D_{T_1}}=E_{-D_1^{(n)}}$
    and $L_{2\Delta_1}^\ast(\overline{\psi_{-D_{T_1}}}, 1)=L_{2}^\ast(\overline{\psi_{-D_1^{(n)}}}, 1)$ holds by Proposition \ref{prop:Tate_alg}.
    When $T_1=\varnothing$, the elliptic curve $E_{-D_{T_1}}=E_{-1}$ has bad reduction at the prime $(1+i)\calO_K$.
    Therefore, we have
    \begin{align}
        L_{2\Delta_1}^\ast(\overline{\psi_{-1}}, 1)
        =L_{\Delta_1}^\ast(\overline{\psi_{-1}}, 1)
        =L^\ast(\overline{\psi_{-1}}, 1)\prod_{i=1}^n\skakko*{1-\frac{\overline{\psi_{-1}}((\pi_{1, i}))}{N(\pi_{1, i})}}.
    \end{align}
    Since $L(\overline{\psi_{-1}}, 1)=\Omega/(2\sqrt{2})$ (\textit{cf}. \cite[p.87]{BirchSwinnerton-Dyer:1965}), we obtain
    \begin{align}
        v_2\skakko*{\frac{L_{2\Delta_1}^\ast(\overline{\psi_{-1}}, 1)}{\Omega}}
        =\sum_{i=1}^nv_2\skakko*{\pi_{1, i}-\skakko*{\dfrac{-1}{\pi_{1, i}}}_4}-\frac{3}{2}\geq n-\frac{3}{2}.\label{eq:Lvaluation1}
    \end{align}
    We prove the theorem by induction on $n$. For $n=1$, by Proposition \ref{prop:finitesumvaluation},
    we see that the $2$-adic valuation of
    \begin{align}
        \frac{L_{2\Delta_1}^{\ast}(\overline{\psi_{-1}}, 1)}{\Omega}+\frac{L_{2}^{\ast}(\overline{\psi_{-D_1^{(n)}}}, 1)}{\Omega}\label{eq:Lvaluation2}
    \end{align}
    is greater than $-1/2$. Since the $2$-adic valuation of the first term in \eqref{eq:Lvaluation2}
    is greater than or equal to $-1/2$ by \eqref{eq:Lvaluation1},
    the valuation of the second term must also be greater than or equal to $-1/2$.
    Thus, it holds for $n=1$. Suppose it is true for $1, \dots, n-1$.
    Then by Proposition \ref{prop:finitesumvaluation}, the $2$-adic valuation of
    \begin{align}
        &\frac{L_{2\Delta_1}^{\ast}(\overline{\psi_{-1}}, 1)}{\Omega}
        +\sum_{\varnothing\neq T_1\subsetneq \{1, \dots, n\}}\frac{L_{2\Delta_1}^{\ast}(\overline{\psi_{-D_{T_1}}}, 1)}{\Omega}
        +\frac{L_{2}^{\ast}(\overline{\psi_{-D_1^{(n)}}}, 1)}{\Omega}\label{eq:Lvaluation3}
    \end{align}
    is greater than $(n-2)/2$. The valuation of the first term in \eqref{eq:Lvaluation3}
    is greater than or equal to $(n-2)/2$ by $\eqref{eq:Lvaluation1}$.
    By using the induction hypothesis, it holds that
    \begin{align}
        &v_2\skakko*{\sum_{\varnothing \neq T_1\subsetneq \{1, \dots, n\}}\dfrac{L_{2\Delta_1}^{\ast}(\overline{\psi_{-D_{T_1}}}, 1)}{\Omega}}\\
        &=v_2\skakko*{\sum_{\varnothing \neq T_1\subsetneq \{1, \dots, n\}}\dfrac{L_{2}^{\ast}(\overline{\psi_{-D_{T_1}}}, 1)}{\Omega}\prod_{\pi_{1, i}\nmid D_{T_1}}\skakko*{1-\dfrac{\overline{\psi_{-D_{T_1}}}((\pi_{1, i}))}{N(\pi_{1, i})}}}\\
        &\geq \min_{\varnothing \neq T_1\subsetneq \{1, \dots, n\}} \mkakko*{\dfrac{\#T_1-2}{2}+\sum_{\pi_{1, i}\nmid D_{T_1}}v_2\skakko*{\pi_{1, i}-\overline{\psi_{-D_{T_1}}}((\pi_{1, i}))}}\\
        &\geq \min_{\varnothing \neq T_1\subsetneq \{1, \dots, n\}}\mkakko*{\dfrac{\# T_1-2}{2}+\dfrac{n-\# T_1}{2}}\\
        &=\dfrac{n-2}{2}.
    \end{align}
    Thus, it also holds for $n$ and we obtain the theorem. 
\end{proof}

\begin{thm}\label{thm:2Lvaluation}
    Under the same conditions as Theorem \ref{thm:Lvaluation}, we have
    \begin{align}
        v_2\skakko*{\dfrac{L_{2}(\overline{\psi_{-D}}, 1)}{\Omega}}\geq
        \begin{cases}
            \dfrac{n+m-2}{2} & (D=D_{1}^{(n)}D_{2}^{(m)}),\vspace{2pt}\\
            \dfrac{m+\ell-2}{2} & (D=D_{2}^{(m)}D_{3}^{(\ell)}),\vspace{2pt}\\
            \dfrac{n+\ell-2}{2} & (D=D_{1}^{(n)}D_{3}^{(\ell)}),\vspace{2pt}\\
            \dfrac{n+m+\ell-2}{2} & (D=D_{1}^{(n)}D_{2}^{(m)}D_{3}^{(\ell)}).
        \end{cases}
    \end{align}
\end{thm}
\begin{proof}
    We only prove the case of $D=D_1^{(n)}D_2^{(m)}$ by double induction on $n$ and $m$ based on the following steps (see FIGURE \ref{fig:Step1andStep2} and FIGURE \ref{fig:Step3}).
    \begin{description}
        \item[Step 1] It holds for $(1, m)$ for all $m$.
        \item[Step 2] It holds for $(n, 1)$ for all $n$.
        \item[Step 3] If it holds for $(n_0, m_0)\neq (n, m) \ (1\leq n_0\leq n, 1\leq m_0\leq m)$, then $(n, m)$ holds.
    \end{description}
    \begin{figure}[htbp]
        \begin{minipage}{0.45\hsize}
            \centering
            \begin{tikzpicture}[scale=0.75]

                \draw[->] (0, 0) -- (6.0, 0); 
                \draw[->] (0, 0) -- (0, 6.0); 

                \draw[very thin, gray] (1, 0) -- (1, 5.9); \node at (1, -0.8) {$1$};
                \draw[very thin, gray] (2, 0) -- (2, 5.9); \node at (2, -0.8) {$2$};
                \node at (3, -0.8) {$\cdots$};
                \draw[very thin, gray] (4, 0) -- (4, 5.9); \node at (4, -0.8) {$n-1$};
                \draw[very thin, gray] (5, 0) -- (5, 5.9); \node at (5, -0.8) {$n$};

                \draw[very thin, gray] (0, 1) -- (5.9, 1); \node at (-0.8, 1) {$1$};
                \draw[very thin, gray] (0, 2) -- (5.9, 2); \node at (-0.8, 2) {$2$};
                \node at (-0.8, 3) {$\vdots$};
                \draw[very thin, gray] (0, 4) -- (5.9, 4); \node at (-0.8, 4) {$m-1$};
                \draw[very thin, gray] (0, 5) -- (5.9, 5); \node at (-0.8, 5) {$m$};

                \fill (1, 1) circle (2pt); \fill (2, 1) circle (2pt); \fill (4, 1) circle (2pt); \fill (5, 1) circle (2pt);
                \fill (1, 2) circle (2pt); \fill (1, 4) circle (2pt); \fill (1, 5) circle (2pt);
                \draw (2, 2) circle [radius=2pt]; \draw (4, 2) circle [radius=2pt]; \draw (5, 2) circle [radius=2pt];
                \draw (2, 4) circle [radius=2pt]; \draw (2, 5) circle [radius=2pt];
                \draw (4, 4) circle [radius=2pt]; \draw (4, 5) circle [radius=2pt];
                \draw (5, 4) circle [radius=2pt]; \draw (5, 5) circle [radius=2pt];

                \draw[rounded corners=5pt] (0.8, 0.8) -- (5.2, 0.8) -- (5.2, 1.2) -- (1.2, 1.2) -- (1.2, 5.2) -- (0.8, 5.2) -- (0.8, 0.8) -- cycle;

            \end{tikzpicture}
            \caption{Step 1 and Step 2}
            \label{fig:Step1andStep2}
        \end{minipage}
        \hspace{20pt}
        \begin{minipage}{0.45\hsize}
            \begin{tikzpicture}[scale=0.75]

                \draw[->] (0, 0) -- (6.0, 0); 
                \draw[->] (0, 0) -- (0, 6.0); 

                \draw[very thin, gray] (1, 0) -- (1, 5.9); \node at (1, -0.8) {$1$};
                \draw[very thin, gray] (2, 0) -- (2, 5.9); \node at (2, -0.8) {$2$};
                \node at (3, -0.8) {$\cdots$};
                \draw[very thin, gray] (4, 0) -- (4, 5.9); \node at (4, -0.8) {$n-1$};
                \draw[very thin, gray] (5, 0) -- (5, 5.9); \node at (5, -0.8) {$n$};

                \draw[very thin, gray] (0, 1) -- (5.9, 1); \node at (-0.8, 1) {$1$};
                \draw[very thin, gray] (0, 2) -- (5.9, 2); \node at (-0.8, 2) {$2$};
                \node at (-0.8, 3) {$\vdots$};
                \draw[very thin, gray] (0, 4) -- (5.9, 4); \node at (-0.8, 4) {$m-1$};
                \draw[very thin, gray] (0, 5) -- (5.9, 5); \node at (-0.8, 5) {$m$};

                \fill (1, 1) circle (2pt); \fill (2, 1) circle (2pt); \fill (4, 1) circle (2pt); \fill (5, 1) circle (2pt);
                \fill (1, 2) circle (2pt); \fill (1, 4) circle (2pt); \fill (1, 5) circle (2pt); \fill (2, 2) circle (2pt);
                \fill (2, 4) circle (2pt); \fill (2, 5) circle (2pt); \fill (4, 2) circle (2pt); \fill (4, 4) circle (2pt);
                \fill (4, 5) circle (2pt); \fill (5, 2) circle (2pt); \fill (5, 4) circle (2pt);
                \draw (5, 5) circle [radius=2pt];

                \draw[rounded corners=5pt] (0.7, 0.7) -- (5.3, 0.7) -- (5.3, 4.3) -- (4.3, 4.3) -- (4.3, 5.3) -- (0.7, 5.3) -- (0.7, 0.7) -- cycle;

                \draw [red, arrows = {-Stealth}] (4.3, 5) -- (4.9, 5);
                \draw [red, arrows = {-Stealth}] (5, 4.3) -- (5, 4.9);

            \end{tikzpicture}
            \caption{Step 3}
            \label{fig:Step3}
        \end{minipage}
    \end{figure}
    First, we show Step 1 by induction on $m$. For $m=1$, the $2$-adic valuation of
    \begin{align}
        \underbrace{\dfrac{L_{2\Delta_1\Delta_2}^\ast(\overline{\psi_{-1}}, 1)}{\Omega}}_{T_1=T_2=\varnothing}
        +\underbrace{\dfrac{L_{2\Delta_1\Delta_2}^\ast(\overline{\psi_{-D_2^{(1)}}}, 1)}{\Omega}}_{T_1=\varnothing, T_2=\{1\}}
        +\underbrace{\dfrac{L_{2\Delta_1\Delta_2}^\ast(\overline{\psi_{-D_1^{(1)}}}, 1)}{\Omega}}_{T_1=\{1\}, T_2=\varnothing}
        +\underbrace{\dfrac{L_{2}^\ast(\overline{\psi_{-D_1^{(1)}D_2^{(1)}}}, 1)}{\Omega}}_{T_1=\{1\}, T_2=\{1\}}\label{eq:2Lvaluation1}
    \end{align}
    is greater than $0$ from Proposition \ref{prop:finitesumvaluation}.
    Therefore, we need to show the first three terms of \eqref{eq:2Lvaluation1}
    is greater than or equal to $0$. For the first term, we see that
    \begin{align}
        v_2\skakko*{\dfrac{L_{2\Delta_1\Delta_2}^\ast(\overline{\psi_{-1}}, 1)}{\Omega}}
        &=v_2\skakko*{\dfrac{L_{2}^\ast(\overline{\psi_{-1}}, 1)}{\Omega}\skakko*{1-\dfrac{\overline{\psi_{-1}}((\pi_{1, 1}))}{N(\pi_{1, 1})}}\skakko*{1-\dfrac{\overline{\psi_{-1}}((\pi_{2, 1}))}{N(\pi_{2, 1})}}}\\
        &\geq -\dfrac{3}{2}+1+1\\
        &>0.
    \end{align}
    For the second term, by Theorem \ref{thm:Lvaluation}, we have
    \begin{align}
        v_2\skakko*{\dfrac{L_{2\Delta_1\Delta_2}^\ast(\overline{\psi_{-D_2^{(1)}}}, 1)}{\Omega}}
        =v_2\skakko*{\dfrac{L_{2}^\ast(\overline{\psi_{-D_2^{(1)}}}, 1)}{\Omega}\skakko*{1-\dfrac{\overline{\psi_{-D_2^{(1)}}}((\pi_{1, 1}))}{N(\pi_{1, 1})}}}
        \geq -\dfrac{1}{2}+1>0.
    \end{align}
    For the third term, we can show that the $2$-adic valuation is greater than $0$ similarly as the second term. Thus it holds for $m=1$. Suppose it is true for $1, \dots, m-1$. Then the 2-adic valuation of
    \begin{equation}
        \begin{split}
            &\underbrace{\dfrac{L_{2\Delta_1\Delta_2}^\ast(\overline{\psi_{-1}}, 1)}{\Omega}}_{T_1=T_2=\varnothing}
            +\sum_{\varnothing \neq T_2\subset\{1, \dots, m\}}\dfrac{L_{2\Delta_1\Delta_2}^\ast(\overline{\psi_{-D_{T_2}}}, 1)}{\Omega}
            +\underbrace{\dfrac{L_{2\Delta_1\Delta_2}^\ast(\overline{\psi_{-D_1^{(1)}}}, 1)}{\Omega}}_{T_1=\{1\}, T_2=\varnothing}\\
            &+\sum_{\varnothing\neq T_2\subsetneq \{1, \dots, m\}}\dfrac{L_{2\Delta_1\Delta_2}^\ast(\overline{\psi_{-D_1^{(1)}D_{T_2}}}, 1)}{\Omega}
            +\underbrace{\dfrac{L_{2}^\ast(\overline{\psi_{-D_1^{(1)}D_2^{(m)}}}, 1)}{\Omega}}_{T_1=\{1\}, T_2=\{1, \dots, m\}}\label{eq:2Lvaluation2}
        \end{split}
    \end{equation}
    is greater than $(m-1)/2$ from Proposition \ref{prop:finitesumvaluation}.
    Therefore, we need to show the first four terms of \eqref{eq:2Lvaluation2}
    is greater than or equal to $(m-1)/2$. For the first term, we see that
    \begin{align}
        v_2\skakko*{\dfrac{L_{2\Delta_1\Delta_2}^\ast(\overline{\psi_{-1}}, 1)}{\Omega}}
        &=v_2\skakko*{\dfrac{L_{2}^\ast(\overline{\psi_{-1}}, 1)}{\Omega}\skakko*{1-\dfrac{\overline{\psi_{-1}}((\pi_{1, 1}))}{N(\pi_{1, 1})}}\prod_{j=1}^{m}\skakko*{1-\dfrac{\overline{\psi_{-1}}((\pi_{2, j}))}{N(\pi_{2, j})}}}\\
        &\geq-\dfrac{3}{2}+1+m\\
        &>\dfrac{m-1}{2}.
    \end{align}
    For the second term, by Theorem \ref{thm:Lvaluation}, we have
    \begin{align}
        &v_2\skakko*{\sum_{\varnothing \neq T_2\subset\{1, \dots, m\}}\dfrac{L_{2\Delta_1\Delta_2}^\ast(\overline{\psi_{-D_{T_2}}}, 1)}{\Omega}}\\
        &\geq \min_{\varnothing\neq T_2\subset\{1, \dots, m\}}\left\{v_2\skakko*{\dfrac{L_{2}^\ast(\overline{\psi_{-D_{T_2}}}, 1)}{\Omega}\skakko*{1-\dfrac{\overline{\psi_{-D_{T_2}}}((\pi_{1, 1}))}{N(\pi_{1, 1})}}\prod_{\pi_{2, j}\nmid D_{T_2}}\skakko*{1-\dfrac{\overline{\psi_{-D_{T_2}}}((\pi_{2, j}))}{N(\pi_{2, j})}}} \right\}\\
        &\geq \min_{\varnothing\neq T_2\subset\{1, \dots, m\}}\left\{\dfrac{2\# T_2-3}{2}+1+(m-\# T_2)\right\}\\
        &>\dfrac{m-1}{2}.
    \end{align}
    For the third term, by Theorem \ref{thm:Lvaluation}, it follows
    \begin{align}
        v_2\skakko*{\dfrac{L_{2\Delta_1\Delta_2}^\ast(\overline{\psi_{-D_1^{(1)}}}, 1)}{\Omega}}
        &=v_2\skakko*{\dfrac{L_{2}^\ast(\overline{\psi_{-D_{1}^{(1)}}}, 1)}{\Omega}\prod_{j=1}^{m}\skakko*{1-\dfrac{\overline{\psi_{-D_{1}^{(1)}}}((\pi_{2, j}))}{N(\pi_{2, j})}}}\\
        &\geq -\dfrac{1}{2}+\dfrac{1}{2}\cdot m\\
        &=\dfrac{m-1}{2}.
    \end{align}
    For the fourth term, by the induction hypothesis, it holds
    \begin{align}
        &v_2\skakko*{\sum_{\varnothing\neq T_2\subsetneq \{1, \dots, m\}}\dfrac{L_{2\Delta_1\Delta_2}^\ast(\overline{\psi_{-D_1^{(1)}D_{T_2}}}, 1)}{\Omega}}\\
        &\geq \min_{\varnothing\neq T_2\subsetneq \{1, \dots, m\}}\left\{v_2\skakko*{\dfrac{L_{2}^\ast(\overline{\psi_{-D_1^{(1)}D_{T_2}}}, 1)}{\Omega}\prod_{\pi_{2, j}\nmid D_{T_2}}\skakko*{1-\dfrac{\overline{\psi_{-D_1^{(1)}D_{T_2}}}((\pi_{2, j}))}{N(\pi_{2, j})}}} \right\}\\
        &\geq \dfrac{1+\# T_2-2}{2}+\dfrac{1}{2}\cdot (m-\# T_2)\\
        &=\dfrac{m-1}{2}.
    \end{align}
    Thus it holds for $m$ and Step 1 is done.

    By a similar calculation, Step 2 can be shown by induction on $n$. We show Step 3. Suppose it is true for $(n_0, m_0) \ (1\leq n_0\leq n, 1\leq m_0\leq m, (n_0, m_0)\neq (n, m))$. Then the $2$-adic valuation of
    \begin{equation}
        \begin{split}
            &\underbrace{\dfrac{L_{2\Delta_1\Delta_2}^{\ast}(\overline{\psi_{-1}}, 1)}{\Omega}}_{T_1=T_2=\varnothing}
            +\sum_{\varnothing\neq T_2\subset\{1, \dots, m\}}\dfrac{L_{2\Delta_1\Delta_2}^{\ast}(\overline{\psi_{-D_{T_2}}}, 1)}{\Omega}
            +\sum_{\varnothing \neq T_1 \subset\{1, \dots, n\}}\dfrac{L_{2\Delta_1\Delta_2}^{\ast}(\overline{\psi_{-D_{T_1}}}, 1)}{\Omega}\\
            &+\sum_{\substack{\varnothing\neq T_1\subsetneq\{1, \dots, n\} \\ \varnothing\neq T_2\subsetneq \{1, \dots, m\}}}\dfrac{L_{2\Delta_1\Delta_2}^{\ast}(\overline{\psi_{-D_{T_1}D_{T_2}}}, 1)}{\Omega}
            +\sum_{\varnothing\neq T_1\subsetneq \{1, \dots, n\}}\dfrac{L_{2\Delta_1\Delta_2}^\ast(\overline{\psi_{-D_{T_1}D_2^{(m)}}}, 1)}{\Omega}\\
            &+\sum_{\varnothing\neq T_2\subsetneq\{1, \dots, m\}}\dfrac{L_{2\Delta_1\Delta_2}^{\ast}(\overline{\psi_{-D_1^{(n)}D_{T_2}}}, 1)}{\Omega}
            +\underbrace{\dfrac{L_{2}^{\ast}(\overline{\psi_{-D_1^{(n)}D_2^{(m)}}}, 1)}{\Omega}}_{T_1=\{1, \dots, n\}, T_2=\{1, \dots, m\}}\label{eq:2Lvaluation3}
        \end{split}
    \end{equation}
    is greater than $(n+m-2)/2$ from Proposition \ref{prop:finitesumvaluation} (see FIGURE \ref{fig:Equation13}).
    \begin{figure}[htbp]
        \centering
        \begin{tikzpicture}[scale=0.8]

            \draw[->] (0, 0) -- (6.0, 0); 
            \draw[->] (0, 0) -- (0, 6.0); 

            \draw[very thin, gray] (1, 0) -- (1, 5.9); \node at (1, -0.8) {$1$};
            \draw[very thin, gray] (2, 0) -- (2, 5.9); \node at (2, -0.8) {$2$};
            \node at (3, -0.8) {$\cdots$};
            \draw[very thin, gray] (4, 0) -- (4, 5.9); \node at (4, -0.8) {$n-1$};
            \draw[very thin, gray] (5, 0) -- (5, 5.9); \node at (5, -0.8) {$n$};

            \draw[very thin, gray] (0, 1) -- (5.9, 1); \node at (-1.0, 1) {$1$};
            \draw[very thin, gray] (0, 2) -- (5.9, 2); \node at (-1.0, 2) {$2$};
            \node at (-1.0, 3) {$\vdots$};
            \draw[very thin, gray] (0, 4) -- (5.9, 4); \node at (-1.0, 4) {$m-1$};
            \draw[very thin, gray] (0, 5) -- (5.9, 5); \node at (-1.0, 5) {$m$};

            \fill (0, 0) circle (2pt);
            \fill (1, 0) circle (2pt); \fill (2, 0) circle (2pt); \fill (4, 0) circle (2pt); \fill (5, 0) circle (2pt);
            \fill (0, 1) circle (2pt); \fill (0, 2) circle (2pt); \fill (0, 4) circle (2pt); \fill (0, 5) circle (2pt);
            \fill (1, 1) circle (2pt); \fill (2, 1) circle (2pt); \fill (4, 1) circle (2pt); \fill (5, 1) circle (2pt);
            \fill (1, 2) circle (2pt); \fill (1, 4) circle (2pt); \fill (1, 5) circle (2pt); \fill (2, 2) circle (2pt);
            \fill (2, 4) circle (2pt); \fill (2, 5) circle (2pt); \fill (4, 2) circle (2pt); \fill (4, 4) circle (2pt);
            \fill (4, 5) circle (2pt); \fill (5, 2) circle (2pt); \fill (5, 4) circle (2pt); \fill (5, 5) circle (2pt);

            \draw[rounded corners=5pt] (-0.3, -0.3) -- (0.3, -0.3) -- (0.3, 0.3) -- (-0.3, 0.3) -- (-0.3, -0.3) -- cycle;
            \draw[rounded corners=5pt] (-0.3, 0.7) -- (0.3, 0.7) -- (0.3, 5.3) -- (-0.3, 5.3) -- (-0.3, 0.7) -- cycle;
            \draw[rounded corners=5pt] (0.7, -0.3) -- (5.3, -0.3) -- (5.3, 0.3) -- (0.7, 0.3) -- (0.7, -0.3) -- cycle;
            \draw[rounded corners=5pt] (0.7, 0.7) -- (4.3, 0.7) -- (4.3, 4.3) -- (0.7, 4.3) -- (0.7, 0.7) -- cycle;
            \draw[rounded corners=5pt] (0.7, 4.7) -- (4.3, 4.7) -- (4.3, 5.3) -- (0.7, 5.3) -- (0.7, 4.7) -- cycle;
            \draw[rounded corners=5pt] (4.7, 0.7) -- (5.3, 0.7) -- (5.3, 4.3) -- (4.7, 4.3) -- (4.7, 0.7) -- cycle;
            \draw[rounded corners=5pt] (4.7, 4.7) -- (5.3, 4.7) -- (5.3, 5.3) -- (4.7, 5.3) -- (4.7, 4.7) -- cycle;

        \end{tikzpicture}
        \caption{Equation \eqref{eq:2Lvaluation3}}
        \label{fig:Equation13}
    \end{figure}

    \noindent Therefore, we need to show the first sixth terms of \eqref{eq:2Lvaluation3} is greater than or equal to $(n+m-2)/2$.
    We calculate the $2$-adic valuation for the first term, second term and fourth term. For the others term, one could calculate similarly. For the first term, we see that
    \begin{align}
        v_2\skakko*{\dfrac{L_{2\Delta_1\Delta_2}^{\ast}(\overline{\psi_{-1}}, 1)}{\Omega}}
        &=v_2\skakko*{\dfrac{L_{2}^{\ast}(\overline{\psi_{-1}}, 1)}{\Omega}\prod_{i=1}^{n}\skakko*{1-\dfrac{\overline{\psi_{-1}}((\pi_{1, i}))}{N(\pi_{1, i})}}\prod_{j=1}^{m}\skakko*{1-\dfrac{\overline{\psi_{-1}}((\pi_{2, j}))}{N(\pi_{2, j})}}}\\
        &\geq-\dfrac{3}{2}+n+m\\
        &>\dfrac{n+m-2}{2}.
    \end{align}
    For the second term, by Theorem \ref{thm:Lvaluation}, it follows
    \begin{align}
        &v_2\skakko*{\sum_{\varnothing\neq T_2\subset\{1, \dots, m\}}\dfrac{L_{2\Delta_1\Delta_2}^{\ast}(\overline{\psi_{-D_{T_2}}}, 1)}{\Omega}}\\
        &\geq \min_{\varnothing\neq T_2\subset\{1, \dots, m\}}\left\{v_2\skakko*{\dfrac{L_{2}^{\ast}(\overline{\psi_{-D_{T_2}}}, 1)}{\Omega}\prod_{i=1}^{n}\skakko*{1-\dfrac{\overline{\psi_{-D_{T_2}}}((\pi_{1, i}))}{N(\pi_{1, i})}}\prod_{\pi_{2, j}\nmid D_{T_2}}\skakko*{1-\dfrac{\overline{\psi_{-D_{T_2}}}((\pi_{2, j}))}{N(\pi_{2, j})}}} \right\}\\
        &\geq \min_{\varnothing\neq T_2\subset\{1, \dots, m\}}\left\{\dfrac{2\# T_2-3}{2}+1\cdot n+1\cdot(m-\# T_2)\right\}\\
        &>\dfrac{n+m-2}{2}.
    \end{align}
    For the fourth term, by the induction hypothesis, it holds
    \footnotesize
    \begin{align}
        &v_2\skakko*{\sum_{\substack{\varnothing\neq T_1\subsetneq\{1, \dots, n\} \\ \varnothing\neq T_2\subsetneq \{1, \dots, m\}}}\dfrac{L_{2\Delta_1\Delta_2}^{\ast}(\overline{\psi_{-D_{T_1}D_{T_2}}}, 1)}{\Omega}}\\
        &\geq \min_{\substack{\varnothing\neq T_1\subsetneq\{1, \dots, n\} \\ \varnothing\neq T_2\subsetneq \{1, \dots, m\}}}\left\{v_2\skakko*{\dfrac{L_{2}^{\ast}(\overline{\psi_{-D_{T_1}D_{T_2}}}, 1)}{\Omega}\prod_{\pi_{1, i}\nmid D_{T_1}}\skakko*{1-\dfrac{\overline{\psi_{-D_{T_1}D_{T_2}}}((\pi_{1, i}))}{N(\pi_{1, i})}}\prod_{\pi_{2, j}\nmid D_{T_2}}\skakko*{1-\dfrac{\overline{\psi_{-D_{T_1}D_{T_2}}}((\pi_{2, j}))}{N(\pi_{2, j})}}} \right\}\\
        &\geq \min_{\substack{\varnothing\neq T_1\subsetneq\{1, \dots, n\} \\ \varnothing\neq T_2\subsetneq \{1, \dots, m\}}}\left\{\dfrac{\# T_1+\# T_2-2}{2}+\dfrac{1}{2}\cdot (n-\# T_1)+\dfrac{1}{2}\cdot (m-\# T_2) \right\}\\
        &=\dfrac{n+m-2}{2}.
    \end{align}
    \normalsize
    Thus it is true for $(n_0, m_0)=(n, m)$ and Step 3 is done. This completes the proof.
\end{proof}

\newpage

\section{Numerical Examples}\label{sec:Example}

We have listed the 2-adic valuation for the case $D=D_1^{(1)}$ and $D=D_1^{(1)}D_2^{(1)}$.
Here, we have arranged it in ascending order of the absolute value of $D$.

\begin{table}[htbp]
    \centering
    \begin{minipage}{0.48\linewidth}
        \centering
        \begin{tabular}{ccc}
            $v_2(L_{2}(\overline{\psi_{-D}}, 1)/\Omega)$ & $D$ & $(i/D)_4$ \\ \hline
            $-1/2$ & $2i - 1$ & $i$ \\
            $0$ & $-3$ & $-1$ \\
            $-1/2$ & $-2i + 3$ & $-i$ \\
            $\infty$ & $-4i + 1$ & $1$ \\
            $-1/2$ & $2i - 5$ & $-i$ \\
            $-1/2$ & $6i - 1$ & $i$ \\
            $0$ & $-4i + 5$ & $-1$ \\
            $\infty$ & $-7$ & $1$ \\
            $-1/2$ & $-2i + 7$ & $i$ \\
            $-1/2$ & $-6i - 5$ & $-i$ \\
            $0$ & $8i - 3$ & $-1$ \\
            $0$ & $8i + 5$ & $-1$ \\
            $1$ & $-4i + 9$ & $1$ \\
            $-1/2$ & $10i - 1$ & $i$ \\
            $-1/2$ & $10i + 3$ & $-i$ \\
            $\infty$ & $-8i - 7$ & $1$ \\
            $0$ & $-11$ & $-1$ \\
            $0$ & $-4i - 11$ & $-1$ \\
            $-1/2$ & $-10i + 7$ & $i$ \\
            $-1/2$ & $6i + 11$ & $-i$ \\
            $-1/2$ & $2i - 13$ & $-i$ \\
            $-1/2$ & $-10i - 9$ & $i$ \\
            $\infty$ & $-12i - 7$ & $1$ \\
            $-1/2$ & $14i - 1$ & $i$ \\
            $-1/2$ & $-2i + 15$ & $i$ \\
            $0$ & $8i + 13$ & $-1$ \\
            $1$ & $-4i - 15$ & $1$ \\
            $\infty$ & $-16i + 1$ & $1$ \\
            $-1/2$ & $-10i - 13$ & $-i$ \\
            $-1/2$ & $-14i - 9$ & $i$ \\
            $0$ & $16i + 5$ & $-1$ \\
            $-1/2$ & $2i - 17$ & $i$ \\
            $0$ & $-12i + 13$ & $-1$ \\
            $-1/2$ & $14i + 11$ & $-i$ \\
            $1$ & $16i + 9$ & $1$ \\
            $-1/2$ & $-18i - 5$ & $-i$ \\
            $\infty$ & $-8i + 17$ & $1$ \\
            $0$ & $-19$ & $-1$ \\
            $-1/2$ & $18i + 7$ & $i$ \\
            $-1/2$ & $10i - 17$ & $i$ \\
        \end{tabular}
    \end{minipage}
    \hspace{10pt}
    \begin{minipage}{0.48\linewidth}
        \centering
        \begin{tabular}{ccc}
            $v_2(L_{2}(\overline{\psi_{-D}}, 1)/\Omega)$ & $D$ & $(i/D)_4$ \\ \hline
            $-1/2$ & $-6i + 19$ & $-i$ \\
            $1$ & $-20i + 1$ & $1$ \\
            $0$ & $20i - 3$ & $-1$ \\
            $-1/2$ & $-14i + 15$ & $i$ \\
            $\infty$ & $-12i + 17$ & $1$ \\
            $\infty$ & $20i - 7$ & $1$ \\
            $0$ & $-4i + 21$ & $-1$ \\
            $-1/2$ & $10i + 19$ & $-i$ \\
            $-1/2$ & $22i - 5$ & $-i$ \\
            $0$ & $-20i - 11$ & $-1$ \\
            $\infty$ & $-23$ & $1$ \\
            $-1/2$ & $10i - 21$ & $-i$ \\
            $-1/2$ & $-14i + 19$ & $-i$ \\
            $0$ & $20i + 13$ & $-1$ \\
            $\infty$ & $-24i + 1$ & $1$ \\
            $\infty$ & $-8i - 23$ & $1$ \\
            $0$ & $-24i + 5$ & $-1$ \\
            $-1/2$ & $-18i - 17$ & $i$ \\
            $0$ & $-16i - 19$ & $-1$ \\
            $1$ & $-4i + 25$ & $1$ \\
            $-1/2$ & $-22i - 13$ & $-i$ \\
            $-1/2$ & $6i - 25$ & $i$ \\
            $\infty$ & $-12i - 23$ & $1$ \\
            $-1/2$ & $26i - 1$ & $i$ \\
            $-1/2$ & $-26i - 5$ & $-i$ \\
            $-1/2$ & $-22i + 15$ & $i$ \\
            $-1/2$ & $-2i + 27$ & $-i$ \\
            $-1/2$ & $26i - 9$ & $i$ \\
            $0$ & $-20i - 19$ & $-1$ \\
            $\infty$ & $-12i + 25$ & $1$ \\
            $-1/2$ & $-22i - 17$ & $i$ \\
            $-1/2$ & $26i + 11$ & $-i$ \\
            $0$ & $28i + 5$ & $-1$ \\
            $-1/2$ & $-14i - 25$ & $i$ \\
            $-1/2$ & $-10i + 27$ & $-i$ \\
            $-1/2$ & $18i + 23$ & $i$ \\
            $0$ & $-4i + 29$ & $-1$ \\
            $-1/2$ & $-6i - 29$ & $-i$ \\
            $1$ & $16i + 25$ & $1$ \\
            $2$ & $20i - 23$ & $1$ \\
        \end{tabular}
    \end{minipage}
    \caption{2-adic valuation for $D=D_1^{(1)}$}
    \label{tab:D1-1}
\end{table}

\begin{table}[htbp]
    \centering
    \begin{minipage}{0.48\linewidth}
        \centering
        \scalebox{0.8}{
        \begin{tabular}{cccc}
            $v_2(L_{2}(\overline{\psi_{-D}}, 1)/\Omega)$ & $D_1^{(1)}$ & $D_2^{(1)}$ & $(i/D)_4$ \\ \hline
            $\infty$ & $-3$ & $(2i - 1)^2$ & $1$ \\
            $0$ & $-2i + 3$ & $(2i - 1)^2$ & $i$ \\
            $0$ & $2i - 1$ & $(-3)^2$ & $i$ \\
            $1$ & $-4i + 1$ & $(2i - 1)^2$ & $-1$ \\
            $1/2$ & $2i - 5$ & $(2i - 1)^2$ & $i$ \\
            $0$ & $2i - 1$ & $(-2i + 3)^2$ & $-i$ \\
            $0$ & $6i - 1$ & $(2i - 1)^2$ & $-i$ \\
            $\infty$ & $-4i + 5$ & $(2i - 1)^2$ & $1$ \\
            $1/2$ & $-2i + 3$ & $(-3)^2$ & $-i$ \\
            $1/2$ & $-7$ & $(2i - 1)^2$ & $-1$ \\
            $\infty$ & $-4i + 1$ & $(-3)^2$ & $1$ \\
            $1/2$ & $2i - 1$ & $(-4i + 1)^2$ & $i$ \\
            $1$ & $-3$ & $(-2i + 3)^2$ & $1$ \\
            $0$ & $2i - 5$ & $(-3)^2$ & $-i$ \\
            $1/2$ & $-3$ & $(-4i + 1)^2$ & $-1$ \\
            $1/2$ & $-4i + 1$ & $(-2i + 3)^2$ & $-1$ \\
            $\infty$ & $6i - 1$ & $(-3)^2$ & $i$ \\
            $1$ & $-11$ & $(2i - 1)^2$ & $1$ \\
            $1/2$ & $-4i + 5$ & $(-3)^2$ & $-1$ \\
            $0$ & $-2i + 3$ & $(-4i + 1)^2$ & $-i$ \\
            $2$ & $-7$ & $(-3)^2$ & $1$ \\
            $\infty$ & $2i - 1$ & $(2i - 5)^2$ & $-i$ \\
            $0$ & $2i - 5$ & $(-2i + 3)^2$ & $i$ \\
            $0$ & $6i - 1$ & $(-2i + 3)^2$ & $-i$ \\
            $0$ & $2i - 1$ & $(6i - 1)^2$ & $-i$ \\
            $1$ & $-4i + 5$ & $(-2i + 3)^2$ & $1$ \\
            $3/2$ & $-3$ & $(2i - 5)^2$ & $1$ \\
            $1/2$ & $-7$ & $(-2i + 3)^2$ & $-1$ \\
            $\infty$ & $2i - 5$ & $(-4i + 1)^2$ & $-i$ \\
            $0$ & $2i - 1$ & $(-4i + 5)^2$ & $i$ \\
            $1$ & $-19$ & $(2i - 1)^2$ & $1$ \\
            $\infty$ & $-11$ & $(-3)^2$ & $-1$ \\
            $1/2$ & $6i - 1$ & $(-4i + 1)^2$ & $i$ \\
            $0$ & $-2i + 3$ & $(2i - 5)^2$ & $i$ \\
            $\infty$ & $-4i + 5$ & $(-4i + 1)^2$ & $-1$ \\
            $0$ & $2i - 1$ & $(-7)^2$ & $i$ \\
            $1$ & $-3$ & $(6i - 1)^2$ & $1$ \\
            $1/2$ & $-23$ & $(2i - 1)^2$ & $-1$ \\
            $3/2$ & $-7$ & $(-4i + 1)^2$ & $1$ \\
            $\infty$ & $-4i + 1$ & $(2i - 5)^2$ & $-1$ \\
        \end{tabular}}
    \end{minipage}
    \hspace{10pt}
    \begin{minipage}{0.48\linewidth}
        \centering
        \scalebox{0.8}{
        \begin{tabular}{cccc}
            $v_2(L_{2}(\overline{\psi_{-D}}, 1)/\Omega)$ & $D_1^{(1)}$ & $D_2^{(1)}$ & $(i/D)_4$ \\ \hline
            $1/2$ & $-3$ & $(-4i + 5)^2$ & $-1$ \\
            $0$ & $-2i + 3$ & $(6i - 1)^2$ & $i$ \\
            $\infty$ & $-11$ & $(-2i + 3)^2$ & $1$ \\
            $\infty$ & $-3$ & $(-7)^2$ & $-1$ \\
            $1/2$ & $-2i + 3$ & $(-4i + 5)^2$ & $-i$ \\
            $1$ & $-4i + 1$ & $(6i - 1)^2$ & $-1$ \\
            $1$ & $-31$ & $(2i - 1)^2$ & $-1$ \\
            $2$ & $-4i + 1$ & $(-4i + 5)^2$ & $1$ \\
            $\infty$ & $-19$ & $(-3)^2$ & $-1$ \\
            $0$ & $6i - 1$ & $(2i - 5)^2$ & $-i$ \\
            $0$ & $-2i + 3$ & $(-7)^2$ & $-i$ \\
            $1$ & $-4i + 5$ & $(2i - 5)^2$ & $1$ \\
            $1/2$ & $-11$ & $(-4i + 1)^2$ & $-1$ \\
            $0$ & $2i - 5$ & $(6i - 1)^2$ & $i$ \\
            $5/2$ & $-4i + 1$ & $(-7)^2$ & $1$ \\
            $1$ & $-7$ & $(2i - 5)^2$ & $-1$ \\
            $\infty$ & $-23$ & $(-3)^2$ & $1$ \\
            $\infty$ & $-43$ & $(2i - 1)^2$ & $1$ \\
            $\infty$ & $2i - 5$ & $(-4i + 5)^2$ & $-i$ \\
            $1/2$ & $-47$ & $(2i - 1)^2$ & $-1$ \\
            $3/2$ & $-4i + 5$ & $(6i - 1)^2$ & $1$ \\
            $3/2$ & $-19$ & $(-2i + 3)^2$ & $1$ \\
            $0$ & $6i - 1$ & $(-4i + 5)^2$ & $i$ \\
            $1$ & $-7$ & $(6i - 1)^2$ & $-1$ \\
            $1/2$ & $2i - 5$ & $(-7)^2$ & $-i$ \\
            $1/2$ & $2i - 1$ & $(-11)^2$ & $i$ \\
            $2$ & $-31$ & $(-3)^2$ & $1$ \\
            $3/2$ & $-7$ & $(-4i + 5)^2$ & $1$ \\
            $1/2$ & $6i - 1$ & $(-7)^2$ & $i$ \\
            $2$ & $-23$ & $(-2i + 3)^2$ & $-1$ \\
            $1/2$ & $-4i + 5$ & $(-7)^2$ & $-1$ \\
            $3/2$ & $-11$ & $(2i - 5)^2$ & $1$ \\
            $1$ & $-19$ & $(-4i + 1)^2$ & $-1$ \\
            $\infty$ & $-3$ & $(-11)^2$ & $-1$ \\
            $\infty$ & $-3$ & $(-11)^2$ & $-1$ \\
            $\infty$ & $-43$ & $(-3)^2$ & $-1$ \\
            $3/2$ & $-23$ & $(-4i + 1)^2$ & $1$ \\
            $1/2$ & $-31$ & $(-2i + 3)^2$ & $-1$ \\
            $1$ & $-11$ & $(6i - 1)^2$ & $1$ \\
            $3$ & $-47$ & $(-3)^2$ & $1$ \\
        \end{tabular}}
    \end{minipage}
    \caption{2-adic valuation for $D=D_1^{(1)}D_2^{(1)}$}
    \label{tab:D1-1D2-1}
\end{table}

\begin{ac}
    The author would like to thank Shinichi Kobayashi
    for his constructive suggestions and thoughtful guidance.
    He is also grateful to Taiga Adachi, Satoshi Kumabe, 
    and Ryota Shii for their discussion and encouragement.
    This work was supported by JST SPRING, Grant Number JPMJSP2136.
\end{ac}

\bibliographystyle{plain}
\bibliography{bibs}

\end{document}